\documentclass[a4paper, 11pt, twoside]{article}

\usepackage{amsmath, amscd, amsfonts, amssymb, amsthm, latexsym, url, color, todonotes, bm, framed, enumerate, afterpage, rotating}
\usepackage{graphicx}
\usepackage[left=1.2in, right=1.2in, top=1.3in, bottom=0.8in, includefoot, headheight=13.6pt]{geometry}
\input{xy}
\xyoption{all}

\usepackage[ps2pdf=true,colorlinks,breaklinks=true]{hyperref}
\usepackage[figure,table]{hypcap}
\hypersetup{
	bookmarksnumbered,
	pdfstartview={FitH},
	citecolor={black},
	linkcolor={black},
	urlcolor={black},
	pdfpagemode={UseOutlines}
}
\makeatletter
\newcommand\org@hypertarget{}
\let\org@hypertarget\hypertarget
\renewcommand\hypertarget[2]{%
  \Hy@raisedlink{\org@hypertarget{#1}{}}#2%
} 
\makeatother 

\newtheorem{theorem}{Theorem}[section]
\newtheorem{lemma}[theorem]{Lemma}
\newtheorem{corollary}[theorem]{Corollary}
\newtheorem{proposition}[theorem]{Proposition}

\theoremstyle{definition}
\newtheorem{definition}[theorem]{Definition}
\newtheorem{remark}[theorem]{Remark}

\newcommand{\xysquare}[8]{
\[\xymatrix{
#1 \ar@{#5}[r] \ar@{#6}[d] & #2 \ar@{#7}[d]\\
#3 \ar@{#8}[r] & #4
}\]
}

\DeclareMathOperator*{\holim}{\operatorname*{holim}}

\newcommand{\al}{\alpha}
\newcommand{\bb}{\mathbb}

\newcommand{\blob}{\bullet}

\newcommand{\cofib}{\rightarrowtail}
\newcommand{\comment}[1]{}

\newcommand{\comp}{{\hat{\phantom{o}}}}

\newcommand{\into}{\hookrightarrow}
\newcommand{\isoto}{\stackrel{\simeq}{\to}}

\newcommand{\onto}{\twoheadrightarrow}
\newcommand{\op}{\operatorname}
\newcommand{\pid}[1]{\langle #1 \rangle}
\newcommand{\quis}{\stackrel{\sim}{\to}}

\newcommand{\roi}{\mathcal{O}}

\newcommand{\sub}[1]{{\mbox{\scriptsize #1}}}

\newcommand{\To}{\longrightarrow}
\newcommand{\ul}[1]{\underline{#1}}

\newcommand{\xto}{\xrightarrow}

\newcommand{\THH}{T\!H\!H}
\newcommand{\HH}{H\!H}
\newcommand{\TR}{T\!R}
\newcommand{\TC}{T\!C}

\renewcommand{\cal}{\mathcal}
\renewcommand{\hat}{\widehat}
\renewcommand{\frak}{\mathfrak}
\newcommand{\indlim}{\varinjlim}

\renewcommand{\projlim}{\varprojlim}

\DeclareMathOperator{\Ext}{Ext}

\DeclareMathOperator{\Hom}{Hom}

\DeclareMathOperator{\Spec}{Spec}

\DeclareMathOperator{\Tor}{Tor}

\newcommand{\dotimes}{\otimes^{\sub{\tiny I\hspace{-1.3mm}L}}}
\newcommand{\xTo}[1]{\stackrel{#1}{\To}}

\usepackage{fancyhdr}

\usepackage{sectsty}
\sectionfont{\Large\sc\centering}
\chapterfont{\large\sc\centering}
\chaptertitlefont{\LARGE\centering}
\partfont{\centering}

\begin{document}

\title{Finite generation and continuity of topological Hochschild and cyclic homology}

\author{Bj\o rn Ian Dundas \& Matthew Morrow}

\date{}

\maketitle

\begin{abstract}
The goal of this paper is to establish fundamental properties of the Hochschild, topological Hochschild, and topological cyclic homologies of commutative, Noetherian rings, which are assumed only to be F-finite in the majority of our results. This mild hypothesis is satisfied in all cases of interest in finite and mixed characteristic algebraic geometry. We prove firstly that the topological Hochschild homology groups, and the homotopy groups of the fixed point spectra $\TR^r$, are finitely generated modules. We use this to establish the continuity of these homology theories for any given ideal. A consequence of such continuity results is the pro Hochschild--Kostant--Rosenberg theorem for topological Hochschild and cyclic homology. Finally, we show more generally that the aforementioned finite generation and continuity properties remain true for any proper scheme over such a ring.
\end{abstract}

\setcounter{tocdepth}{2}
\tableofcontents

\setcounter{section}{-1}

\newpage
\section{Introduction}
The aim of this paper is to prove fundamental finite generation, continuity, and pro Hochschild--Kostant--Rosenberg theorems for the Hochschild, topological Hochschild, and topological cyclic homologies of commutative, Noetherian rings. As far as we are aware, these are the first general results on finite generation and continuity of topological \linebreak Hochschild and cyclic homology, despite the obvious foundational importance of such problems in the subject.

The fundamental hypothesis for the majority of our theorems is the classical notion of F-finiteness:

\begin{definition}\label{definition_F-finite_intro}
A $\bb Z_{(p)}$-algebra (always commutative) is said to be {\em F-finite} if and only if the $\bb F_p$-algebra $A/pA$ is finitely generated over its subring of $p^\sub{th}$-powers.
\end{definition}

This is a mild condition: it is satisfied as soon as $A/pA$ is obtained from a perfect field by iterating the following constructions any finite number of times: passing to finitely generated algebras, localising, completing, or Henselising; see Lemma \ref{lemma_conditions_for_F-finite}.

To state our main finite generation result, we first remark that the Hochschild homology $\HH_n(A)$ of a ring $A$ is always understood in the derived sense (see Section \ref{subsection_AQ_HH}). Secondly, $\THH_n(A)$ denotes the topological Hochschild homology groups of a ring $A$, while $\TR^r_n(A;p)$ denotes, as is now usual, the homotopy groups of the fixed point spectrum for the action of the cyclic group $C_{p^{r-1}}$ on the topological Hochschild homology spectrum $\THH(A)$. It is known that $\TR^r_n(A;p)$ is naturally a module over the $p$-typical Witt ring $W_r(A)$. The obvious notation will be used for the $p$-completed, or finite coefficient, versions of these theories, and for their extensions to quasi-separated, quasi-compact schemes following \cite{GeisserHesselholt1999}.

Our main finite generation result is the following, where $A_p^\comp=\projlim_sA/p^sA$ denotes the $p$-completion of a ring $A$:

\begin{theorem}[see Corol.~\ref{corollary_finite_gen_of_THH_TR_p_complete}]\label{theorem_intro_finite_gen_of_THH_TR}
Let $A$ be a Noetherian, F-finite $\bb Z_{(p)}$-algebra. Then:
\begin{enumerate}
\item $\HH_n(A;\bb Z_p)$ and $\THH_n(A;\bb Z_p)$ are finitely generated $A_p^\comp$-modules for all $n\ge0$.
\item $\TR_n^r(A;p,\bb Z_p)$ is a finitely generated $W_r(A_p^\comp)$-module for all $n\ge 0$, $r\ge1$.
\end{enumerate}
\end{theorem}

The key step towards proving Theorem \ref{theorem_intro_finite_gen_of_THH_TR} is the following finite generation result for the Andr\'e--Quillen homology of $\bb F_p$-algebras:

\begin{theorem}[see Thm.~\ref{theorem_F_finite_implies_AQ_finite}]
Let $A$ be a Noetherian, F-finite $\bb F_p$-algebra. Then the Andr\'e--Quillen homology groups $D_n^i(A/\bb F_p)$ are finitely generated for all $n,i\ge0$.
\end{theorem}

Next we turn to ``degree-wise continuity'' for the homology theories $\HH$, $\THH$, and $\TR^r$, by which we mean the following: given an ideal $I\subseteq A$, we examine when the natural map of pro $A$-modules \[\{\HH_n(A)\otimes_AA/I^s\}_s\To\{\HH_n(A/I^s)\}_s\] is an isomorphism, and similarly for $\THH$ and $\TR^r$. This question was first raised by L.~Hesselholt in 2001 \cite{Carlsson2001}, who later proved with T.~Geisser the $\THH$ isomorphism in the special case that $A=R[X_1,\dots,X_d]$ and $R=\pid{X_1,\dots,X_s}$ for any ring $R$ \cite[\S1]{GeisserHesselholt2006b}.

In Section \ref{subsection_continuity} we prove the following:

\begin{theorem}[see Thm.~\ref{theorem_continuity}]\label{theorem_intro_degreewise_continuity}
Let $A$ satisfy the same conditions as Theorem \ref{theorem_intro_finite_gen_of_THH_TR}, and let $I\subseteq A$ be an ideal. Then the following canonical maps of pro $A$-modules
\begin{itemize}
\item[] $\{\HH_n(A;\bb Z/p^v)\otimes_AA/I^s\}_s\To\{\HH_n(A/I^s;\bb Z/p^v)\}_s$
\item[] $\{\THH_n(A;\bb Z/p^v)\otimes_AA/I^s\}_s\To\{\THH_n(A/I^s;\bb Z/p^v)\}_s$
\end{itemize}
and the following canonical map of pro $W_r(A)$-modules
\begin{itemize}
\item[] $\{\TR_n^r(A;p,\bb Z/p^v)\otimes_{W_r(A)}W_r(A/I^s)\}_s\To\{\TR_n^r(A/I^s;p,\bb Z/p^v)\}_s$
\end{itemize}
are isomorphisms for all $n\ge0$ and $r,v\ge1$.
\end{theorem}

Applying Theorem \ref{theorem_intro_degreewise_continuity} simultaneously to $A$ and its completion $\hat A=\projlim_sA/I^s$ with respect to the ideal $I$, we obtain the following corollary:

\begin{corollary}[see Corol.~\ref{corollary_pre_continuity}]\label{corollary_intro_completing}
Let $A$ be a Noetherian, F-finite $\bb Z_{(p)}$-algebra, and $I\subseteq A$ an ideal. Then all of the following canonical maps (not just the compositions) are isomorphisms for all $n\ge0$ and $r,v\ge1$:
\begin{align*}
&\HH_n(A;\bb Z/p^v)\otimes_A\hat A\To \HH_n(\hat A;\bb Z/p^v)\To \projlim_s\HH_n(A/I^s;\bb Z/p^v)\\
&\THH_n(A)\otimes_A\hat A\To \THH_n(\hat A)\To \projlim_s\THH_n(A/I^s;\bb Z/p^v)\\
&\TR^r_n(A;p,\bb Z/p^v)\otimes_{W_r(A)}W_r(\hat A)\To \TR^r_n(\hat A;p,\bb Z/p^v)\To\projlim_s\TR^r_n(A/I^s;p,\bb Z/p^v)
\end{align*}
\end{corollary}

Of a more topological nature than such degree-wise continuity statements are spectral continuity, namely the question of whether the canonical map of spectra \[\THH(A)\To\holim_s\THH(A/I^s)\] is a weak-equivalence, at least after $p$-completion. The analogous continuity question for $K$-theory was studied for discrete valuation rings by A.~Suslin \cite{Suslin1984a}, I.~Panin \cite{Panin1986}, and the first author \cite{Dundas1998}, and for power series rings $A=R[[X_1,\dots,X_d]]$ over a regular, F-finite $\bb F_p$-algebra $R$ by Geisser and Hesselholt \cite{GeisserHesselholt2006b}, with $I=\pid{X_1,\dots,X_d}$. Geisser and Hesselholt proved the continuity of $K$-theory in such cases by establishing it first for $\THH$ and $\TR^r$.

We use the previous degree-wise continuity results to prove the following:

\begin{theorem}[see Thm.~\ref{theorem_continuity_in_complete_case}]\label{theorem_intro_continuity_in_complete_case}
Let $A$ be a Noetherian, F-finite $\bb Z_{(p)}$-algebra, and $I\subseteq A$ an ideal; assume that $A$ is $I$-adically complete. Then the following canonical maps are weak equivalences after $p$-completion:
\begin{align*}
&\THH(A)\To\holim_s\THH(A/I^s)\\
&\TR^r(A;p)\To\holim_s\TR^r(A/I^s;p)\quad(r\ge1)
\end{align*}
and similarly for $\TR(-;p)$, $\TC^r(-;p)$, and $\TC(-;p)$.
\end{theorem}

There are two important special cases in which the results so far can be analysed further: when $p$ is nilpotent, and when $p$ generates $I$. Firstly, if $p$ is nilpotent in $A$, for example if $A$ is a Noetherian, F-finite, $\bb F_p$-algebra, then Theorem \ref{theorem_intro_finite_gen_of_THH_TR} -- Theorem \ref{theorem_intro_continuity_in_complete_case} are true integrally, without $p$-completing or working with finite coefficients; see Corollaries \ref{corollary_finite_gen_nilp} and \ref{corollary_nilpotent_continuity} for precise statements. The second important special case is $I=pA$, when Theorems \ref{theorem_intro_degreewise_continuity} and \ref{theorem_intro_continuity_in_complete_case}  simplify to the following statements:

\begin{corollary}[see Corol.~\ref{corollary_p_I}]\label{corollary_intro_p_I}
Let $A$ be a Noetherian, F-finite $\bb Z_{(p)}$-algebra. Then the following canonical maps are isomorphisms for all $n\ge0$ and $v,r\ge1$:
\begin{itemize}
\item[] $\HH_n(A;\bb Z/p^v)\To\{\HH_n(A/p^sA;\bb Z/p^v)\}_s$
\item[] $\THH_n(A;\bb Z/p^v)\To\{\THH_n(A/p^sA;\bb Z/p^v)\}_s$
\item[] $\TR_n^r(A;p,\bb Z/p^v)\To\{\TR_n^r(A/p^sA;p,\bb Z/p^v)\}_s$ 
\item[] $\TC_n^r(A;p,\bb Z/p^v)\To\{\TC_n^r(A/p^sA;p,\bb Z/p^v)\}_s$ 
\end{itemize}
Moreover, all of the following maps (not just the compositions) are weak equivalences after $p$-completion:
\begin{itemize}
\item[] $\THH(A)\To\THH(A_p^\comp)\To\holim_s\THH(A/p^sA)$
\item[] $\TR^r(A;p)\To\TR^r(A_p^\comp;p)\To\holim_s\TR^r(A/p^sA;p)\quad(r\ge1)$,
\end{itemize}
and similarly for $\TR(-;p)$, $\TC^r(-;p)$, and $\TC(-;p)$.
\end{corollary}

Corollary \ref{corollary_intro_p_I} was proved by Hesselholt and I.~Madsen \cite[Thm.~5.1]{Hesselholt1997} for finite algebras over the Witt vectors of perfect fields, and then by Geisser and Hesselholt \cite[\S3]{GeisserHesselholt2006a} for any (possibly non-commutative) ring $A$ in which $p$ is a non-zero divisor.  Our result eliminates the assumption that $p$ be a non-zero divisor, but at the expense of requiring $A$ to be a Noetherian, F-finite $\bb Z_{(p)}$-algebra.

Next we present our pro Hochschild--Kostant--Rosenberg theorems, which can be found in Section \ref{subsection_pro_HKR}, starting with algebraic Hochschild homology. Given a geometrically regular (e.g., smooth) morphism $k\to A$ of Noetherian rings, the classical Hochschild--Kostant--Rosenberg theorem states that the antisymmetrisation map $\Omega_{A/k}^n\to HH_n^k(A)$ is an isomorphism of $A$-modules for all $n\ge0$. We establish its pro analogue in full generality:

\begin{theorem}[see Thm.~\ref{theorem_pro_HKR}]\label{theorem_intro_pro_HKR}
Let $k\to A$ be a geometrically regular morphism of Noetherian rings, and let $I\subseteq A$ be an ideal. Then the canonical map of pro $A$-modules \[\{\Omega^n_{(A/I^s)/k}\}_s\To\{HH_n^k(A/I^s)\}_s\] is an isomorphism for all $n\ge0$.
\end{theorem}

In the special case of finite type algebras over fields, the theorem was proved by G.~Corti\~nas, C.~Haesemeyer, and C.~Weibel \cite[Thm.~3.2]{Cortinas2009}. The stronger form presented here has recently been required in the study of infinitesimal deformations of algebraic cycles \cite{BlochEsnaultKerz2013, Morrow_Deformational_Hodge}.

The analogue of the classical Hochschild--Kostant--Rosenberg theorem for topological Hochschild homology and the fixed point spectra $\TR^r$ were established by Hesselholt \cite[Thm.~B]{Hesselholt1996} and state the following: If $A$ is a regular $\bb F_p$-algebra, then there is a natural isomorphism of $W_r(A)$-modules \[\bigoplus_{i=0}^nW_r\Omega_A^i\otimes_{W_r(\bb F_p)}TR_{n-i}^r(\bb F_p;p)\isoto  TR_n^r(A;p),\] where $W_r\Omega_A^*$ denotes the de~Rham--Witt complex of S.~Bloch, P.~Deligne, and L.~Illusie. In the limit over $r$, Hesselholt moreover showed that the contribution from the left side vanishes except in top degree $i=n$, giving an isomorphism of pro abelian groups $\{W_r\Omega_A^i\}_r\cong\{\TR^r_n(A;p)\}_r$ which deserves to be called the Hochschild--Kostant--Rosenberg theorem for the pro spectrum $\{\TR^r\}_r$.

We prove the following pro versions of these HKR theorem with respect to powers of an ideal:

\begin{theorem}[see Thm.~\ref{theorem_top_pro_HKR} \& Corol.~\ref{corollary_top_pro_HKR_for_TR}]
Let $A$ be a regular, F-finite $\bb F_p$-algebra, and $I\subseteq A$ an ideal. Then:
\begin{enumerate}
\item The canonical map \[\bigoplus_{i=0}^n\{W_r\Omega_{A/I^s}^i\otimes_{W_r(\bb F_p)}\TR^r_{n-i}(\bb F_p;p)\}_s\To \{\TR^r_n(A/I^s;p)\}_s\] of pro $W_r(A)$-modules is an isomorphism for all $n\ge0$, $r\ge1$.
\item The canonical map of pro pro abelian groups, i.e., in the category $\op{Pro}(\op{Pro}Ab)$,
\[\big\{\{W_r\Omega_{A/I^s}^n\}_s\big\}_r\To \big\{\{\TR^r_n(A/I^s;p)\}_s\big\}_r\] is an isomorphism for all $n\ge0$.
\end{enumerate}
\end{theorem}

Finally, in Section \ref{section_proper}, the earlier finite generation and continuity results are extended from Noetherian, F-finite $\bb Z_{(p)}$-algebras to proper schemes over such rings. These are obtained by combining the results in the affine case with Zariski descent and Grothendieck's formal functions theorem for coherent cohomology. Our main finite generation result in this context is the following:

\begin{theorem}[see Corol.~\ref{corollary_p_complete_finite_gen_proper}]\label{theorem_intro_fin_gen_p_complete}
Let $A$ be a Noetherian, F-finite, finite Krull-dimensional $\bb Z_{(p)}$-algebra, and $X$ a proper scheme over $A$. Then:
\begin{enumerate}
\item $\HH_n(X;\bb Z_p)$ and $\THH_n(X;\bb Z_p)$ are finitely generated $A_p^\comp$-modules for all $n\ge~0$.
\item $\TR_n^r(X;p,\bb Z_p)$ is a finitely generated $W_r(A_p^\comp)$-module for all $n\ge0$, $r\ge1$.
\end{enumerate}
\end{theorem}

Next we consider analogues of degree-wise continuity for proper schemes over $A$: given an ideal $I\subseteq A$ and a proper scheme $X$ over $A$, we consider the natural map of pro $A$-modules \[\{\HH_n(X)\otimes_AA/I^s\}_s\To\{\HH_n(X_s)\}_s\] where $X_s:=X\times_AA/I^s$, and similarly for $\THH$, $\TR$. In this setup we establish the exact analogues of Theorem \ref{theorem_intro_degreewise_continuity} and Corollary \ref{corollary_intro_completing}; see Theorem \ref{theorem_degree_wise_cont_proper} and Corollary \ref{corollary_pre_continuity_proper} for precise statements.

As in the affine case, we use these degree-wise continuity statements to deduce continuity of topological cyclic homology for proper schemes over our usual base rings:

\begin{theorem}[see Thm.~\ref{theorem_spectral_continuity_proper}]
Let $A$ be a Noetherian, F-finite $\bb Z_{(p)}$-algebra, $I\subseteq A$ an ideal, and $X$ be a proper scheme over $A$; assume that $A$ is $I$-adically complete. Then the following canonical maps of spectra are weak equivalences after $p$-completion:
\begin{align*}
&\THH(X)\To\holim_s\THH(X_s)\\
&\TR^r(X;p)\To\holim_s\TR^r(X_s;p)\quad(r\ge1),
\end{align*}
and similarly for $\TR(-;p)$, $\TC^r(-;p)$, and $\TC(-;p)$.
\end{theorem}

Also as in the affine case, it is particular interesting to consider the cases that $p$ is nilpotent or is the generator of $I$; see Corollaries \ref{corollary_p_nilpotent_proper} and \ref{corollary_proper_I_p}.

\subsection*{Notation, etc.}
All rings from Section \ref{section_Witt} onwards are tacitly understood to be commutative; the strict exception is Proposition \ref{proposition_Bjorns_SS}, which holds also in the non-commutative case. Modules are therefore understood to be symmetric bimodules whenever a bimodule structure is required for Hochschild homology; again, the strict exception is Proposition \ref{proposition_Bjorns_SS}, where a non-symmetric module naturally appears. 

Given a positive integer $n$, the $n$-torsion of an abelian group $M$ is denoted by $M[n]$.

\subsection*{Acknowledgements}
The second author would like to thank the Department of Mathematics at the University of Bergen for providing such a hospitable environment during two visits.

\section{Review of Artin--Rees properties and homology theories}\label{section_review}

\subsection{Pro abelian groups and Artin--Rees properties}
Here we summarise some results and notation concerning pro abelian group and pro modules which will be used throughout the paper.

If $\cal A$ is a category, then we denote by $\op{Pro}\cal A$ the category of pro objects of $\cal A$ indexed over the natural numbers. That is, an object of $\op{Pro}\cal A$ is an inverse system \[M_\infty=\{M_s\}_s=``M_1\leftarrow M_2\leftarrow\cdots",\] where the objects $M_i$ and the transition maps belong to $\cal A$; the morphisms are given by \[\Hom_{\op{Pro}\cal A}(M_\infty,N_\infty):=\projlim_r\indlim_s\Hom_\cal A(M_s,N_r).\]

If $\cal A$ is abelian then so is $\op{Pro}\cal A$, and a pro object $M_\infty\in\op{Pro}\cal A$ is isomorphic to zero if and only if it satisfies the trivial Mittag-Leffler condition: that is, that for all $r\ge1$ there exists $s\ge r$ such that the transition map $M_s\to M_r$ is zero.

We will be particularly interested in the cases $\cal A=Ab$ and $A\op{-mod}$, where $A$ is a commutative ring, in which case we speak of pro abelian groups and pro $A$-modules respectively. When it is unlikely to cause any confusion, we will occasionally use $\infty$ notation in proofs for the sake of brevity; for example, if $I$ is an ideal of a ring $A$ and $M$ is an $A$-module, then \[M\otimes_AA/I^\infty=\{M\otimes_AA/I^s\}_s,\quad\HH_n(A/I^\infty,M/I^\infty M)=\{\HH_n(A/I^s,M/I^sM)\}_s.\]

For the sake of reference, we now formally state the fundamental Artin--Rees result which will be used repeatedly; this result appears to have been first noticed and exploited by M.~Andr\'e \cite[Prop.~10 \& Lem.~11]{Andre1974} and D.~Quillen \cite[Lem.~9.9]{Quillen1968}:

\begin{theorem}[Andr\'e, Quillen]\label{theorem_Artin_Rees}
Let $A$ be a commutative, Noetherian ring, and $I\subseteq A$ an ideal.
\begin{enumerate}
\item If $M$ is a finitely generated $A$-module, then the pro $A$-module $\{\Tor_n^A(A/I^s,M)\}_s$ vanishes for all $n>0$.
\item The ``completion'' functor
\begin{align*}
-\otimes_A A/I^\infty\,:A\op{-mod}&\To  \op{Pro}A\op{-mod}\\
M&\mapsto \{M\otimes_AA/I^s\}_s
\end{align*}
is exact on the subcategory of finitely generated $A$-modules.
\end{enumerate}
\end{theorem}
\begin{proof}[Sketch of proof]
By picking a resolution $P_\bullet$ of $M$ by finitely generated projective $A$-modules and applying the classical Artin--Rees property to the pair $d(P_n)\subseteq P_{n-1}$, one sees that for each $r\ge1$ there exists $s\ge r$ such that the map \[\Tor_n^A(A/I^s,M)\to\Tor_n^A(A/I^r,M)\] is zero. This proves (i). (ii) is just a restatement of (i).
\end{proof}

\begin{corollary}\label{corollary_Artin_Rees}
Let $A$ be a commutative, Noetherian ring, $I\subseteq A$ an ideal, and $M$ a finitely generated $A$-module. Then the pro $A$-module $\{\Tor_n^A(A/I^s,M/I^s)\}_s$ vanishes for all $n>0$.
\end{corollary}
\begin{proof}
For each $r\ge1$ we may apply the previous theorem to the module $M/I^r$ to see that there exists $s\ge r$ such that the second of the following arrows is zero: \[\Tor_n^A(A/I^s,M/I^s)\to \Tor_n^A(A/I^s,M/I^r)\to \Tor_n^A(A/I^r,M/I^r)\] Hence the composition is zero, completing the proof.
\end{proof}

\begin{corollary}\label{corollary_Artin_Rees_group_homology}
Let $A$ be a commutative, Noetherian ring, $I\subseteq A$ an ideal, $M$ a finitely generated $A$-module, and $G$ a finite group acting $A$-linearly on $M$. Then the canonical map of pro group homologies \[\{H_n(G,M)\otimes_AA/I^s\}_s\To \{H_n(G,M/I^sM)\}_s\] is an isomorphism for all $n\ge0$.
\end{corollary}
\begin{proof}
Considering $\bb Z$ as a left $\bb Z[G]$-module via the augmentation map, $A/I^s$ as a right $A$-module, and $M$ as an $A-\bb Z[G]$-bimodule, there are first quadrant spectral sequences of $A$-modules with the same abutement by \cite[Ex.~5.6.2]{Weibel1994}: \[E_{ij}^2(s)=\Tor_i^A(A/I^s,\Tor_j^{\bb Z[G]}(M,\bb Z)),\quad ^\prime E_{ij}^2(s)=\Tor_i^{\bb Z[G]}(\Tor_j^A(A/I^s,M),\bb Z).\] These assemble to first quadrant spectral sequences of pro $A$-modules with the same abutement: \[E_{ij}^2(\infty)=\{\Tor_i^A(A/I^s,\Tor_j^{\bb Z[G]}(M,\bb Z))\}_s,\quad ^\prime E_{ij}^2(\infty)=\{\Tor_i^{\bb Z[G]}(\Tor_j^A(A/I^s,M),\bb Z)\}_s.\]

Since $\Tor^{\bb Z[G]}_j(M,\bb Z)$ is a finitely generated $A$-module for all $j\ge0$, Theorem \ref{theorem_Artin_Rees}(i) implies that $E_{ij}^2(\infty)=0$ for $i>0$; so the $E(\infty)$-spectral sequence degenerates to edge map isomorphisms. Theorem \ref{theorem_Artin_Rees}(i) similarly implies that $\Tor_j^A(A/I^s,M)=0$ for $j>0$, and hence the $^\prime E(\infty)$-spectral sequence also degenerates to edge map isomorphisms.

Composing these edge map isomorphisms, we arrive at an isomorphism of pro $A$-modules \[\{\Tor^{\bb Z[G]}_n(M,\bb Z)\otimes_AA/I^s\}_s\isoto \{\Tor^{\bb Z[G]}_n(M/I^sM,\bb Z)\}_s\] for all $n\ge0$, which is exactly the desired isomorphism.
\end{proof}

\begin{corollary}\label{corollary_pro_torsion}
Let $A$ be a commutative, Noetherian ring, $I\subseteq A$ an ideal, $M$ a finitely generated $A$-module, and $m\ge1$. Then the canonical maps \[\{M[m]\otimes_AA/I^s\}_s\To\{M/I^sM\,[m]\}_s,\quad \{M/mM\otimes_AA/I^s\}_s\To\{M/(mM+I^sM)\}_s\] are isomorphisms of pro $A$-modules.
\end{corollary}
\begin{proof}
These isomorphisms follow from  the sequence \[0\to\{M[m]\otimes_AA/I^s\}_s\to\{M/I^sM\}\xto{\times m}\{M/I^sM\}\to\{M/mM\otimes_AA/I^s\}_s\to0,\] which is exact by Theorem \ref{theorem_Artin_Rees}(ii).
\end{proof}

\subsection{Andr\'e--Quillen and Hochschild homology}\label{subsection_AQ_HH}
Let $k$ be a commutative ring. Here we briefly review the Andr\'e--Quillen and Hochschild homologies of $k$-algebras, though we assume that the reader has some familiarity with these theories.

We begin with Andr\'e--Quillen homology \cite{Andre1974, Quillen1970, Ronco1993}. Let $A$ be a commutative $k$-algebra, let $P_\bullet\to A$ be a simplicial resolution of $A$ by free commutative $k$-algebras, and set \[\bb L_{A/k}:= \Omega_{P_\bullet/k}^1\otimes_{P_\bullet}A.\] Thus $\bb L_{A/k}$ is a simplicial $A$-module which is free in each degree; it is called the cotangent complex of the $k$-algebra $A$. Up to homotopy, the cotangent complex depends only on $A$, since the free simplicial resolution $P_\bullet\to A$ is unique up to homotopy.

Given simplicial $A$-modules $M_\bullet$, $N_\bullet$, the tensor product and alternating powers are the simplicial $A$-modules defined degreewise: $(M_\bullet\otimes_AN_\bullet)_n=M_n\otimes_AN_n$ and $(\bigwedge_A^rM_\bullet)_n=\bigwedge_A^rM_n$.

In particular we set $\bb L_{A/k}^i:=\bigwedge_A^i\bb L_{A/k}$ for each $i\ge 1$. The Andr\'e--Quillen homology of the $k$-algebra $A$, with coefficients in an $A$-module $M$, is defined by \[D_n^i(A/k,M):=\pi_n(\bb L_{A/k}^i\otimes_A M).\tag{$n,i\ge0$}\] When $M=A$ the notation is simplified to $D_n^i(A/k):=D_n^i(A/k,A)=\pi_n\bb L_{A/k}^i$.
\comment{
If $k\to A$ is essentially of finite type and $k$ is Noetherian, then $D_n^i(A/k,M)$ is a finitely generated $A$-module for all $n,i$ and for all finitely generated $A$-modules $M$.

If $0\to M\to N\to P\to 0$ is a short exact sequence of $A$-modules, then there is a resulting long exact sequence for each $i\ge 1$: \[\cdots\To D_n^i(A/k,M)\To D_n^i(A/k,N)\To D_n^i(A/k,P)\To \cdots\]

Finally, if $k\to A\to B$ are ring homomorphisms then the simplicial resolutions may be chosen so that there is an exact sequence of simplicial $B$-modules \[0\To\bb L_{A/k}\otimes_A B\To\bb L_{B/k}\To\bb L_{B/A}\To 0.\] This remains exact upon tensoring by any $B$-module $M$ since these simplicial $B$-modules are free in each degree; taking homology yields the Jacobi--Zariski long exact sequence \[\cdots\To D_n(A/k,M)\To D_n(B/k,M)\To D_n(B|A,M)\To\cdots\]
}

Now we discuss Hochschild homology \cite{Loday1992}, so let $A$ be a possibly non-commutative $k$-algebra (however, apart from Proposition \ref{proposition_Bjorns_SS}, all rings from Section \ref{section_Witt} onwards are commutative). For an $A$-bimodule $M$, the ``usual'' Hochschild homology of $A$ as a $k$-algebra with coefficients in $M$ is defined to be $\HH_n^{\sub{usual},k}(A,M):=\pi_n(C_\bullet^k(A,M))$ for $n\ge0$, where $C_\blob^k(A,M)$ is the usual simplicial $k$-module
\[C_\bullet^k(A,M):=\qquad M
\begin{array}{c}
\longleftarrow\\[-7pt]
\longleftarrow
\end{array}
M\otimes_kA
\begin{array}{c}
\longleftarrow\\[-7pt]
\longleftarrow\\[-7pt]
\longleftarrow
\end{array}
M\otimes_kA\otimes_kA
\begin{array}{c}
\longleftarrow\\[-7pt]
\longleftarrow\\[-7pt]
\longleftarrow\\[-7pt]
\longleftarrow
\end{array}
\cdots
\]
However, we will work throughout with the derived version of Hochschild homology, which we now explain; more details may be found in \cite[\S4]{Morrow_birelative}. Letting $P_\bullet\to A$ be a simplicial resolution of $A$ by free $k$-algebras, let $\HH^k(A,M)$ denote the diagonal of the bisimplicial $k$-module $C_\blob^k(P_\blob,M)$; the homotopy type of $\HH^k(A,M)$ does not depend on the choice of resolution, and we set \[\HH_n^k(A,M):=\pi_n\HH^k(A,M).\tag{$n\ge0$}\] We stress that there is a canonical map $\HH_n^k(A,M)\to \HH_n^{\sub{usual},k}(A,M)$ for all $n\ge0$, which is an isomorphism if $A$ is flat over $k$ (in particular, if $k$ is a field).

The advantage of derived, rather than usual, Hochschild homology is that it ensures the existence of two spectral sequences: firstly the Andr\'e--Quillen-to-Hochschild-homology spectral sequence when $A$ is commutative, \[E^2_{ij}=D_i^j(A/k,M)\Longrightarrow \HH_{i+j}^k(A,M),\] and secondly the Pirashvili--Waldhausen spectral sequence which will be discussed in Section \ref{subsection_THH}. In the special case $A=M$ we write $\HH^k(A)=\HH^k(A,A)$ and $\HH_n^k(A)=\HH_n^k(A,A)$, and when $k=\bb Z$ we omit it from the notation.

\subsection{Topological Hochschild and cyclic homology}\label{subsection_intro_to_THH}
The manipulations of topological Hochschild and cyclic homology contained in this paper are of a mostly algebraic nature, using only the formal properties of the theory. In this section we explicitly collect various recurring spectral sequences, long exact sequences, etc.~which we need; we hope that the algebraic nature of this exposition will be accessible to non-topologists since the results of this paper will be later applied to problems in arithmetic and algebraic geometry.

\subsubsection{Topological Hochschild homology}\label{subsection_THH}
If $A$ is a ring and $M$ is an $A$-bimodule then $\THH(A,M)$ denotes the associated topological Hochschild homology spectrum, as constructed in, e.g. \cite{Dundas2013}. Its homotopy groups are the topological Hochschild homology of $A$ with coefficients in $M$, namely \[\THH_n(A,M):=\pi_n\THH(A,M).\tag{$n\ge0$}\] If $A$ is commutative and $M$ is a symmetric $A$-module, then these are $A$-modules. In the most important case of $M=A$, one writes \[\THH(A)=\THH(A,A),\quad\quad \THH_n(A)=\THH_n(A,A).\]

Algebraic properties of the groups $\THH_n(A,M)$ may frequently be extracted from the following two results:
\begin{enumerate}
\item The Pirashvili--Waldhausen \cite[Thm.~4.1]{PirashviliWaldhausen1992} \cite[Lem.~4.2.3.7]{Dundas2013} first quadrant spectral sequence of abelian groups (of $A$-modules if $A$ is commutative and $M$ is a symmetric $A$-module) \[E^2_{ij}=\HH_i(A,\THH_j(\bb Z,M))\implies \THH_{i+j}(A,M),\] which compares the topological Hochschild homology groups with their algebraic counterpart $\HH_n(A,M)=\HH_n^\bb Z(A,M)$.
\item M.~B\"okstedt's \cite[Thm.~4.1.0.1]{Dundas2013} calculation of the groups $\THH_n(\bb Z,M)$:
\[\THH_n(\bb Z,M)\cong\begin{cases}M&n=0\\\Tor^\bb Z_0(\bb Z/m\bb Z,M)= M/mM&n=2m-1\\\Tor^\bb Z_1(\bb Z/m\bb Z,M)={M[m]}&n=2m>0.\end{cases}\]
\end{enumerate}
For example, if $A$ is a commutative, Noetherian ring for which $\HH_n(A)$ is a finitely generated $A$-module for all $n\ge 0$, then (i) \& (ii) easily imply the same of the $A$-modules $\THH_n(A)$; this will be used in the proof of Theorem \ref{theorem_finite_gen_of_THH_TR}.

\subsubsection{The fixed point spectra $TR^r$}\label{subsection_TR}
An essential fact for the foundations of and calculations in topological cyclic homology is that $\THH(A)$ is not merely a spectrum, but a cyclotomic spectrum in the sense of \cite{Hesselholt1997}. This means that $\THH(A)$ is an $S^1$-spectrum (i.e.,~it carries a ``nice'' action by the circle group $S^1$), which is already enough to ensure the existence of so-called Frobenius and Verschiebung maps, together with additional pro structure ensuring the existence of restriction maps. The intricacies of cyclotomic spectra and the construction of their homotopy fixed points, homotopy orbits, etc.~are irrelevant for this paper; we only need certain algebraic consequences which we now list.

Let $p$ be a fixed prime number. For $r\ge 1$, one lets $C_{p^r}$ denote the cyclic subgroup of $S^1$ of order $p^r$, and one denotes by $\TR^r(A;p):=\THH(A)^{C_{p^{r-1}}}$ the $C_{p^{r-1}}$-fixed point spectrum of $\THH(A)$, and by $TR_n^r(A;p)$ its homotopy groups: \[\TR_n^r(A;p):=\pi_n\TR^r(A;p)=\pi_n(\THH(A)^{C_{p^{r-1}}}).\] Note that $\TR_n^1(A;p)=\THH_n(A)$.

Formal algebraic properties of the groups $\TR_n^r(A;p)$ may be obtained from the following two facts, which are non-trivial consequences of $\THH(A)$ being a cyclotomic spectrum; see, e.g., Lems.~1.4.5 \& 2.0.6 of \cite[\S VI]{Dundas2013}, or Thm.~1.2 and the proof of Prop.~2.3 of \cite{Hesselholt1997}:
\begin{enumerate}
\item There is a natural homotopy fibre sequence  of spectra \[\THH(A)_{hC_{p^r}}\To \TR^{r+1}(A;p)\xTo{R} \TR^r(A;p)\] where $\THH(A)_{hC_{p^r}}$ is the spectrum of homotopy orbits for the action of $C_{p^r}$ on $\THH(A)$. This gives rise to a long exact sequence of homotopy groups \[\cdots\To\pi_n(\THH(A)_{hC_{p^r}})\To \TR^{r+1}_n(A;p)\xTo{R} \TR^r_n(A;p)\To\cdots\] $R$ is known as the {\em restriction map}.
\item The homotopy groups of $\THH(A)_{hC_{p^r}}$ appearing in the previous long exact sequence are described by a first quadrant spectral sequence of abelian groups \[E^2_{ij}=H_i(C_{p^r},\THH_j(A))\implies \pi_{i+j}(\THH(A)_{hC_{p^r}}),\] where the groups on the $E^2$-page are group homology for $C_{p^r}$ acting trivially on $\THH_j(A)$.
\end{enumerate}

\subsubsection{Witt structure}\label{subsection_Witt_structure}
Assuming that $A$ is commutative, the various aforementioned groups inherit natural module structures:
\begin{enumerate}\itemsep0mm
\item $\THH_n(A)$ is a module over $A$.
\item $\TR^r_n(A;p)$ and $\pi_n(\THH(A)_{hC_{p^r}})$ are modules over the $p$-typical Witt vectors $W_r(A)$.
\end{enumerate}
The Witt vector structure appears in the following way: firstly, $\TR^r(A;p)$ is a ring spectrum and so the homotopy groups $\TR_n^r(A;p)$ and $\pi_n(\THH(A)_{hC_{p^r}})$ are modules over the ring $TR_0^r(A;p)$ (this does not require $A$ to be commutative), and secondly it is a theorem of Hesselholt and Madsen \cite[Thm.~F]{Hesselholt1997} that there is a natural isomorphism of rings $W_r(A)\isoto TR_0^r(A;p)$. Moreover, by \cite[\S1.3]{Hesselholt1996} we have the following structure:
\begin{enumerate}\itemsep0mm
\item The long exact sequence of homotopy groups from \ref{subsection_TR}(i) above is a long exact sequence of $W_{r+1}(A)$-modules, where the action of $W_{r+1}(A)$ on the $W_r(A)$-module $TR_n^r(A;p)$ is via the restriction map $R:W_{r+1}(A)\to W_r(A)$.
\item The group homology spectral sequence from \ref{subsection_TR}(ii) above is a spectral sequence of $W_{r+1}(A)$-modules, where the action of $W_{r+1}(A)$ on the $E^2$-page, whose terms are clearly $A$-modules, is via the $r^\sub{th}$ power of the Frobenius $F^r:W_{r+1}(A)\to A$.
\end{enumerate}

\subsubsection{Topological cyclic homology}
As indicated in \ref{subsection_TR}(i) above, there is a restriction map of spectra $R:\TR^{r+1}(A;p)\to \TR^r(A;p)$. Using these as transition maps, one puts \[\TR(A;p)=\TR^\infty(A;p):=\op{holim}_rTR^r(A;p),\] which is a ring spectrum whose homotopy groups $TR_n(A;p):=\pi_n(TR(A;p))$ fit into short exact sequences \[0\To{\projlim_r}^1\TR_{n+1}^r(A;p)\To TR_n(A;p)\To\projlim_rTR_n^r(A;p)\To 0.\] If $A$ is commutative then the groups $TR_n(A;p)$ are naturally modules over $W(A)=\projlim_rW_r(A)$.

To reach topological cyclic homology, one must finally introduce the so-called {\em Frobenius map} $F:\TR^{r+1}(A;p)\to \TR^r(A;p)$, which is nothing other than the  inclusion of the $C_{p^r}$-fixed point spectrum of $\THH(A)$ into the $C_{p^{r-1}}$-fixed point spectrum. The Frobenius commutes with the restriction, and thus induces a map $F:\TR(A;p)\to \TR(A;p)$. The $p$-typical topological cyclic homology spectrum of $A$ is, by definition, \[\TC(A;p)=\TC^\infty(A;p):=\op{hofib}(\TR(A;p)\xto{\op{id}-F} \TR(A;p)),\] and thus its homotopy groups $\TC_n(A;p):=\pi_n\TC(A;p)$ fit into a long exact sequence \[\cdots\To \TC_n(A;p)\To \TR_n(A;p)\xto{\op{id}-F} \TR_n(A;p)\To\cdots\]

We remark that the notation $\TR^\infty(A;p)=\TR(A;p)$ and $\TC^\infty(A;p)=\TC(A;p)$ is not standard, but we adopt it to more succulently state a number of our results.

One may additionally define a $p$-typical topological cyclic homology spectrum for a fixed level $r\ge1$ by setting \[\TC^r(A;p):=\op{hofib}(\TR^r(A;p)\xto{R-F} \TR^r(A;p)),\] whose homotopy groups therefore fit into a long exact sequence \[\cdots\To \TC_n^r(A;p)\To \TR_n^r(A;p)\xto{R-F} \TR_n^{r-1}(A;p)\To\cdots\] The commutativity of homotopy limits implies that there is a natural weak equivalence $\TC(A;p)\quis\op{holim}_r\TC^r(A;p)$.

\subsection{Finite coefficients and $p$-completions}\label{subsection_p_completions}
It will be necessary at times to work with versions of the aforementioned homology theories with finite coefficients or after $p$-completing. Here we will review some standard machinery and notation. Additional lemmas concerning $p$-completions may be found in the appendix.

 Given a prime number $p$ and a simplicial abelian group $M_\blob$, its {\em (derived) $p$-completion} is by definition the simplicial abelian group \[(M_{\blob})_p^\comp:=\holim_v(M_\blob\dotimes_{\bb Z}\bb Z/p^v\bb Z).\] The homotopy groups of $M\dotimes_{\bb Z}\bb Z/p^v\bb Z$ are denoted by $\pi_n(M_\blob;\bb Z/p^v)$ and fit into an exact sequence \[0\To \pi_n(M_\blob)\otimes_\bb Z\bb Z/p^v\bb Z\To\pi_n(M_\blob;\bb Z/p^v)\To\pi_{n-1}(M_\blob)[p^v]\To0\] Also, the homotopy group of the derived $p$-completion are denoted  by $\pi_n(M_\blob;\bb Z_p)$ and fit into an exact sequence \[0\to\op{Ext}^1_{\bb Z}(\bb Q_p/\bb Z_p,\pi_n(M_\blob))\to \pi_n(M_\blob;\bb Z_p)\to\Hom_{\bb Z}(\bb Q_p/\bb Z_p,\pi_{n-1}(M_\blob))\to 0.\]
 
Similarly, if $X$ is a spectrum then its {\em $p$-completion} is by definition \[X_p^\comp:=\holim_v(X\wedge S/p^v),\] where $S/p^r$ denotes the $p^{r\,th}$ Moore spectrum; the same short exact sequences as for a simplicial abelian group apply, and we point out that $H((M_\blob)_p^\comp)=H(M_\blob)_p^\comp$ for any simplicial abelian group $M_\blob$, where $H(-)$ denotes the Eilenberg--Maclane construction.

We remark that if $M$ is an abelian group then $M_p^\comp$ denotes the usual $p$-adic completion of $M$, namely $M_p^\comp=\projlim_vM/p^vM$, and not the derived $p$-completion of $M$ as a constant simplicial abelian group.

Now let $A$ be a commutative ring, and $M$ an $A$-module. We will use the notation
\[\begin{array}{ll}
\HH(A,M;\bb Z/p^v):=\HH(A,M)\dotimes_{\bb Z}\bb Z/p^v\bb Z & \HH(A,M;\bb Z_p):=\HH(A,M)_p^\comp\\
\THH(A,M;\bb Z/p^v):=\THH(A,M)\wedge S/p^v & \THH(A,M;\bb Z_p):=\THH(A,M)_p^\comp\\
\TR^r(A;\bb Z/p^v):=\TR^r(A;p)\wedge S/p^v & \TR^r(A;p,\bb Z_p):=\TR(A;p)_p^\comp\\
\end{array}
\]
and similarly for $\TR$, $\TC^r$, and $\TC$; the homotopy groups are denoted in the obvious manner. To make the already overburdened notation more manageable, we have chosen to write $\TC^r(A;\bb Z/p^v)$, $\TC^r(A;\bb Z/p^v)$, etc., rather than $\TC^r(A;p,\bb Z/p^v)$,  $\TC^r(A;p,\bb Z/p^v)$,~etc.

There is a exact sequence of simplicial $A$-modules $0\to \Sigma A[p^v]\to A\dotimes_{\bb Z}\bb Z/p^v\bb Z\to A/p^vA\to 0$ (where $\Sigma$ denotes suspension, i.e., $-1$ shift in the language of complexes), and this induces a long exact sequence of $A$-modules \[\cdots\To\HH_{n-1}(A,A[p^v])\To\HH_n(A;\bb Z/p^v)\To\HH_n(A,A/p^vA)\To\cdots\] This will be a useful tool for deducing results for $\HH_n(A;\bb Z/p^v)$ via $\HH$ with coefficients in $A$-modules. There is an analogous long exact sequence for $\THH$.

We now make some comments about how the properties of topological Hochschild and cyclic homology from Section \ref{subsection_intro_to_THH} respect finite coefficients. Smashing the homotopy fibre sequence of Section \ref{subsection_TR}(i) with $S/p^v$ yields a new homotopy fibre sequence, and hence a long exact sequence of the homotopy groups with finite coefficients. Moreover, $\THH(A)_{hC_{p^r}}\wedge S/p^v\simeq (\THH(A)\wedge S/p^v)_{hC_{p^r}}$, and hence there is a homotopy orbit spectral sequence, as in Section \ref{subsection_TR}(ii), with finite coefficients.

Next, $\HH(A;\bb Z/p^v)$ is a simplicial module over $\HH(A)$; $\THH(A;\bb Z/p^v)$ is a module spectrum over $\THH(A)$; and $\TR^r(A;\bb Z/p^v)$ is a module spectrum over $\TR^r(A;p)$. Hence their homotopy groups are naturally modules over $A$, $A$, and $W_r(A)$ respectively. The Witt structure outlined in Section \ref{subsection_Witt_structure} thus remains true with finite coefficients.

\section{Preliminaries on Witt rings and F-finiteness}\label{section_Witt}
\subsection{Witt rings}\label{subsection_Witt}
In this section we establish some preliminary results on Witt rings, which will be required throughout the paper; some similar material may be found in \cite{GeisserHesselholt2006}. In particular, in Theorem \ref{theorem_base_change_by_Witt_rings} we establish a technical relationship between completing along a pro Witt ring $W_r(A/I^\infty)$ and the Frobenius map.

We begin with a reminder on Witt rings of a ring $A$ and associated notation; recall that all rings are henceforth commutative. We will use the language of both big Witt rings $\bb W_S(A)$ associated to truncation sets $S\subseteq\bb N$, and of $p$-typical Witt rings $W_r(A)=\bb W_{\{1,p,p^2,\dots,p^{r-1}\}}(A)$ when a particular prime number $p$ is clear from the context. The $p$-typical case is classical, while the language of truncation sets is due to \cite{Hesselholt2001}; a more detailed summary, to which we will refer for various Witt ring identities, may be found in \cite[Appendix]{Rulling2007}.

Given an inclusion of truncation sets $S\supseteq T$, there are associated Restriction, Frobenius and Verschiebung maps \[R_T,\;F_T:\bb W_S(A)\to\bb W_T(A),\quad V_T:\bb W_T(A)\to\bb W_S(A).\] The Restriction $R_T$ and the Frobenius $F_T$ are ring homomorphisms, while $V_T$ is merely additive. If $m\ge1$ is an integer then one defines a new truncation set by \[S/m:=\{s\in S:sm\in S\}\] and writes $R_m$, $F_m$, and $V_m$ instead of $R_{S/m}$, $F_{S/m}$, and $V_{S/m}$ respectively.

The Teichm\"uller map $[-]_S:A\to\bb W_S(A)$ is multiplicative. If $S$ is finite then each element of $\bb W_S(A)$ may be written uniquely as $\sum_{i\in S}V_i[a_i]_{S/i}$ for some $a_i\in A$; we will often use this to reduce questions to the study of terms of the form $V_i[a]_{S/i}$, which we will abbreviate to $V_i[a]$ when the truncation set $S$ is clear.

In the $p$-typical case we follow the standard abuse of notation, writing $R,\; F:W_r(A)\to W_{r-1}(A)$ and $V:W_{r-1}(A)\to W_r(A)$ in place of $R_p,F_p$, and $V_p$.

If $I$ is an ideal of $A$ then $\bb W_S(I)$ denotes the ideal of $\bb W_S(A)$ defined as the kernel of the quotient map $\bb W_S(A)\onto\bb W_S(A/I)$. Alternatively, $\bb W_S(I)$ is the Witt vectors of the non-unital ring $I$. An element $\al\in \bb W_S(A)$ lies in $\bb W_S(I)$ if and only if, in its expansion $\al=\sum_{i\in S}V_i[a_i]$, the coefficients $a_i\in A$ all belong to $I$. 

The following lemma establishes the first collection of basic properties we need concerning the ideals $\bb W_S(I)$:

\begin{lemma}\label{lemma_witt_1}
Let $A$ be a ring, $I,J\subseteq A$ ideals, and $S$ a finite truncation set. Then:
\begin{enumerate}
\item $\bb W_S(I)\bb W_S(J)\subseteq\bb W_S(IJ)$.
\item $\bb W_S(I)^N\subseteq\bb W_S(I^N)$ for all $N\ge1$.
\item $\bb W_S(I)+\bb W_S(J)=\bb W_S(I+J)$.
\item Assume $I$ is a finitely generated ideal; then for any $N\ge 1$ there exists $M\ge 1$ such that $\bb W_S(I^M)\subseteq\bb W_S(I)^N$.
\end{enumerate}
\end{lemma}
\begin{proof}
(i): It is enough to show that $\al\beta\in\bb W_S(IJ)$ in the case that $\al=V_i[a]$ and $\beta=V_j[b]$ for some $i,j\in S$, $a\in I$, and $b\in J$, since such terms additively generate $\bb W_S(I)$ and $\bb W_S(J)$. But this follows from the standard Witt ring identity \cite[A.4(v)]{Rulling2007} \[V_i[a]\,V_j[b]=gV_{ij/g}[a^{i/g}b^{j/g}],\] where $g:=\gcd(i,j)$. Now (ii) follows from (i) by induction.

(iii): The surjection $J\to\tfrac{I+J}{I}$ induces a surjection \[\bb W_S(J)\onto\bb W_S(\tfrac{I+J}{I})\cong\tfrac{\bb W_S(I+J)}{\bb W_S(I)},\] whence $\bb W_S(I+J)\subseteq\bb W_S(I)+\bb W_S(J)$. The reverse inclusion is obvious.

(iv): By assumption we have $I=\pid{t_1,\dots,t_m}$ for some $t_1,\dots,t_m\in A$. For any $M\ge 1$, we will write $I^{(M)}:=\pid{t_1^M,\dots,t_m^M}\subseteq I^M$. Note that $I^{(M)}\supseteq I^{m(M-1)+1}$, so it is enough to find $M\ge1$ such that $\bb W_S(I^{(M)})\subseteq \bb W_S(I)^N$; we claim that $M=N\ell$ suffices, where $\ell$ is the least common multiple of all elements of $S$. To prove the claim we first use (iii) to see that $\bb W_S(I^{(M)})=\bb W_S(At_1^M)+\cdots+\bb W_S(At_m^M)$, and then we note that $\bb W_S(At_j^M)$ is additively generated by terms $V_i[at_j^M]$ where $i\in S$ and $a\in A$; so it is enough to prove the claim for such terms. Writing $M=N\ell=N'i$ for some $N'\ge N$, this claim follows from the standard Witt ring identity \cite[A.4(vi)]{Rulling2007} \[V_i[at_j^M]=V_i[at_j^{N'i}]=[t_j]^{N'}V_i[a]\in\bb W_S(I)^{N'}\subseteq \bb W_S(I)^N.\]
\end{proof}

\begin{remark}\label{remark_improved_witt_1}
The proof of part (iv) of Lemma \ref{lemma_witt_1} establishes a stronger result: namely that for any $N\ge1$ and any set of generators $t_1,\dots,t_m$ of $I$, there exists $M\ge1$ such that \[\bb W_S(I^M)\subseteq \pid{[t_1],\dots,[t_m]}^N.\] In particular, if $f:A\to B$ is a ring homomorphism, then we have $\bb W_S(f(I^M)B)\subseteq f(\bb W_S(I)^N)\bb W_S(B)$.
\comment{
 for any ring homomorphism $f:A\to B$. Indeed, choosing $M$ as in the proof above, it is required to show that $V_i[bf(t_j^M)]\in f(\bb W_S(I^N))\bb W_S(B)$ for any $i\in S$ and $b\in B$. This follows as it did in the previous proof, from the calculation \[V_i[bf(t_j^M)]=V_i[bf(t_j)^{N'i}]=[f(t_j)]^{N'}V_i[b]=f([t_j^{N'}])V_i[b].\]
}
\end{remark}

The next two lemmas establish the basic relationship between completions of Witt rings and Witt rings of completions:

\begin{lemma}\label{lemma_W_S_of_completion}
Let $A$ be a ring, $I\subseteq A$ a finitely generated ideal, and $S$ a finite truncation set; let $\hat A:=\projlim_sA/I^s$ be the $I$-adic completion of $A$. Then the canonical maps \[\projlim_s\bb W_S(A)/\bb W_S(I)^s\To\projlim_s\bb W_S(A)/\bb W_S(I^s)\longleftarrow \bb W_S(\hat A)\] are isomorphisms.
\end{lemma}
\begin{proof}
The left arrow is an isomorphism since the two chains of ideals $\bb W_S(I^s)$ and $\bb W_S(I)^s$ are intertwined by Lemma \ref{lemma_witt_1}. It remains to consider the right arrow.

For $S=\{1\}$ there is nothing to prove since $\bb W_{\{1\}}(A)=A$ and $\bb W_{\{1\}}(I)=I$. We proceed by induction on the maximal element $m$ of the truncation set $S$; putting $T:=S\setminus\{m\}$ and noticing that $S/m=\{1\}$, we have a short exact sequence \[0\To\bb W_{\{1\}}(R)\xto{V_m}\bb W_S(R)\xto{R_T}\bb W_T(R)\To0\] for any ring $R$, possibly non-unital. We thus arrive at the following comparison of short exact sequences, where $\hat A$ denotes the $I$-adic completion of $A$:
\[
\xymatrix{
0\ar[r]&\bb W_{\{1\}}(\hat A)\ar[r]^{V_m}\ar[d]&\bb W_S(\hat A)\ar[r]^{R_T}\ar[d]&\bb W_T(\hat A)\ar[r]\ar[d]&0\\
0\ar[r]&\projlim_s\bb W_{\{1\}}(A/I^s)\ar[r]^{V_m}&\projlim_s\bb W_S(A/I^s)\ar[r]^{R_T}&\projlim_s\bb W_T(A/I^s)\ar[r]&0
}\]
Applying the inductive hypotheses to $T$, we see that the left and right vertical arrows are isomorphisms, whence the central is too.
\end{proof}

In the remainder of the section we fix a prime number $p$ and focus on the $p$-typical Witt rings, starting with a $p$-adic analogue of the previous completion lemma. In the following lemma, as well as later in the paper, the $p$-adic completion of a ring $R$ is denoted by $R_p^\comp:=\projlim_sR/p^sR$.

\begin{lemma}\label{lemma_W_r_of_p_completion}
Let $A$ be a ring, $p$ a prime number, and $r\ge1$. Then there is a natural isomorphism of rings $W_r(A)_p^\comp\cong W_r(A_p^\comp)$.
\end{lemma}
\begin{proof}
By Lemma \ref{lemma_W_S_of_completion}, it is enough to show that the ideals $pW_r(A)$ and $W_r(pA)$ each contain a power of the other.

Firstly, it is well-known that $W_r(\bb F_p)=\bb Z/p^r\bb Z$, whence $W_r(A/pA)$ is a $\bb Z/p^r\bb Z$-algebra; in other words, $p^rW_r(A)\subseteq W_r(pA)$.

Secondly, by Remark \ref{remark_improved_witt_1} there exists $M\ge1$ such that $W_r(p^MA)\subseteq[p]^pW_r(A)$; so \[W_r(pA)^M\subseteq W_r(p^MA)\subseteq [p]^pW_r(A),\] where the first inclusion is by Lemma \ref{lemma_witt_1}(ii). Hence we can complete the proof by showing that $[p]^p\in pW_r(A)$. Since $R^{r-1}([p])=p\in A$, and since there is a short exact sequence \[0\To W_{r-1}(A)\xto{V}W_r(A)\xto{R^{r-1}}A\To 0,\] we may write $[p]-p=V(\al)$ for some $\al\in W_{r-1}(A)$. Raising to the $p^\sub{th}$ power, we deduce that $[p]^p=p\beta+V(\al)^p$ for some $\beta\in W_r(A)$, and so it is enough to show that $V(\al)^2\in pW_r(A)$. But indeed it follows from standard Witt vector identities that the square of the ideal $VW_{r-1}(A)$ lies inside $pW_r(A)$, e.g., \cite[Prop.~A.4(v)]{Rulling2007}.
\end{proof}

Now we turn to the Frobenius:

\begin{lemma}\label{lemma_witt_2}
Let $A$ be a ring, $I\subseteq A$ a finitely generated ideal, and $r\ge1$.
\begin{enumerate}
\item The ideal of $A$ generated by $F^{r-1}W_r(I)$ contains $I^M$ for $M\gg 0$.
\item The natural ring homomorphisms $A\otimes_{W_r(A)}W_r(A/I^s)\to A/I^s$, induced by the commutative diagrams of rings
\[\xymatrix{
W_r(A) \ar[r]\ar[d]_{F^{r-1}} & W_r(A/I^s)\ar[d]^{F^{r-1}}\\
A \ar[r] & A/I^s,
}\] induce an isomorphism of pro rings $\{A\otimes_{W_r(A)}W_r(A/I^s)\}_s\isoto \{A/I^s\}_s$.
\end{enumerate}
\end{lemma}
\begin{proof}
(i): If $I$ is generated by $t_1,\dots,t_m\in A$, then $I^{(M)}=\pid{t_1^M,\dots,t_m^M}$ contains $I^{m(M-1)+1}$; so it is enough to show that $I^{(M)}\subseteq\pid{F^{r-1}W_r(I)}$ for $M\gg 0$. But $M=p^{r-1}$ clearly has this property, since for any $a\in I$ we have $F^{r-1}(a,0,\dots,0)=a^{p^{r-1}}$.

(ii): Since $W_r(A)\to W_r(A/I^s)$ is surjective with kernel $W_r(I^s)$, the tensor product $A\otimes_{W_r(A)}W_r(A/I^s)$ is simply $A/\pid{F^{r-1}W_r(I^s)}$. Thus the claimed isomorphism of pro rings is the statement that the chains of ideals $I^s$ and $F^{r-1}W_r(I^s)A$, for $s\ge1$, are intertwined; one inclusion is obvious and the other is (i).
\end{proof}

To say more in the $p$-typical we will focus on $\bb Z_{(p)}$-algebras $A$ which are F-finite; recall from Definition \ref{definition_F-finite_intro} that this means $A/pA$ is finitely generated over its subring of $p^\sub{th}$-powers. We will require the following results of A.~Langer and T.~Zink, which may be found in the appendix of \cite{LangerZink2004}, concerning Witt vectors of such rings:

\comment{
\begin{definition}\label{definition_F_finite}
An $\bb F_p$-algebra $A$ is called {\em F-finite} if and only if it is finitely generated as an algebra (or equivalently, as a module) over its subring $A^p$ of $p^\sub{th}$-powers. More generally, a $\bb Z_{(p)}$-algebra $A$ is called F-finite if and only if $A/pA$ is F-finite in the previous sense.

For example, a perfect field $k$ of characteristic $p$ is F-finite since $k^p=k$; to construct more examples from this, see Lemma \ref{lemma_conditions_for_F-finite}.
\end{definition}
}

\begin{theorem}[Langer--Zink]\label{theorem_Langer_Zink}
Let $A$ be an F-finite $\bb Z_{(p)}$-algebra and let $r\ge1$. Then:
\begin{enumerate}
\item The Frobenius $F:W_{r+1}(A)\to W_r(A)$ is a finite ring homomorphism; i.e., $W_r(A)$ is finitely generated as a module over its subring $FW_{r+1}(A)$.
\item If $B$ is a finitely generated $A$-algebra, then $B$ is also F-finite and $W_r(B)$ is a finitely generated $W_r(A)$-algebra.
\item If $A$ is Noetherian then $W_r(A)$ is also Noetherian.
\end{enumerate}
\end{theorem}

Combining these results of Langer--Zink with our own lemmas, we reach the main result of this section, in which we relate the Frobenius on $W_r(A)$ with pro completion along an ideal of $A$; this will be the primary algebraic tool by which we inductively extend results from $\THH$ to $\TR^r$:

\begin{theorem}\label{theorem_base_change_by_Witt_rings}
Let $A$ be a Noetherian, F-finite $\bb Z_{(p)}$-algebra, $I\subseteq A$ an ideal, and $r\ge1$. Consider the ``completion'' functor
\[\xymatrix@R=1mm@C=3cm{
\Phi_r:W_r(A)\op{-mod}\ar[r]^{-\otimes_{W_r(A)}W_r(A/I^\infty)} &  \op{Pro}W_r(A)\op{-mod}\\
M\ar@{|->}[r] & \{M\otimes_{W_r(A)}W_r(A/I^s)\}_s
}\]
Then:
\begin{enumerate}
\item $\Phi_r$ is exact on the subcategory of finitely generated $W_r(A)$-modules.
\item There is a natural equivalence of functors between the composition \[A\op{-mod}\xto{(F^{r-1})^*}W_r(A)\op{-mod}\xto{\Phi_r}\op{Pro}W_r(A)\op{-mod}\] and the composition \[A\op{-mod}\xto{\Phi_1}\op{Pro}A\op{-mod}\xto{(F^{r-1})^*}\op{Pro}-W_r(A)\op{-mod}\] In other words, if $M$ is an $A$-module, viewed as a $W_r(A)$-module via the $r-1$ power of the Frobenius $F^{r-1}:W_r(A)\to A$, then there is a natural isomorphism \[\Phi_r(A)\cong \{M\otimes_AA/I^s\}_s.\]
\end{enumerate}
\end{theorem}
\begin{proof}
(i): We must prove that the pro abelian group $\{\op{Tor}_n^{W_r(A)}(W_r(A/I^s), M)\}_s$ vanishes for any finitely generated $W_r(A)$-module $M$ and integer $n>0$.

According to Lemma \ref{lemma_witt_1}(ii)+(iv), the chain of ideals $W_r(I^s)$ is intertwined with the chain $W_r(I)^s$, so it is sufficient to prove that the pro abelian group \[\{\op{Tor}_n^{W_r(A)}(W_r(I)^s, M)\}_s\] vanishes. But according to Langer--Zink $W_r(A)$ is Noetherian, so this vanishing claim is covered by the Artin--Rees theorem recalled in Theorem \ref{theorem_Artin_Rees}(i).

(ii) is a restatement of Lemma \ref{lemma_witt_2}(ii).
\end{proof}

\subsection{F-finiteness}\label{subsection_F-finiteness}
In this section we prove some basic properties surrounding F-finiteness (Definition \ref{definition_F-finite_intro}), for which we claim no originality but for which we know of no suitable reference. We fix a prime number $p$ for the next three lemmas.

Our main results apply to Noetherian, F-finite $\bb Z_{(p)}$-algebras, so we first show that this property is preserved under many natural constructions:

\begin{lemma}\label{lemma_conditions_for_F-finite}
Let $A$ be a Noetherian, F-finite $\bb Z_{(p)}$-algebra; then the following are also Noetherian, F-finite $\bb Z_{(p)}$-algebras:
\begin{enumerate}
\item Any finitely generated $A$-algebra.
\item Any localisation of $A$ at a mutiplicative system.
\item The completion of $A$ at any ideal.
\item The Henselisation of $A$ at any ideal.
\item The strict Henselisation of $A$ at any maximal ideal.
\end{enumerate}
\end{lemma}
\begin{proof}
The claim that operations (i)--(v) preserve the Noetherian property is standard commutative algebra, so we need only prove the F-finiteness assertion, for which we may replace $A$ by $A/pA$ and therefore assume that $A$ is an $\bb F_p$-algebra. Let $a_1,\dots,a_n$ generate $A$ as an algebra over its $p^\sub{th}$-powers $A^p$.

(i): If $B=A[b_1,\dots,b_m]$ is a finitely generated $A$-algebra, then it is generated by $a_1,\dots,a_n,b_1,\dots,b_m$ over its $p^\sub{th}$-powers, hence is F-finite. (ii): If $B$ is a localisation of $A$ then $B$ is also generated by $a_1,\dots,a_n$ over its $p^\sub{th}$-powers. (iii): Let $\hat A$ be the completion of $A$ at an ideal $I\subseteq A$. Picking generators $t_1,\dots,t_d\in I$ for the ideal $I$, there is a resulting surjection $A[[X_1,\dots,X_d]]\to \hat A$ and so the F-finiteness of $\hat A$ follows from that of $A[[X_1,\dots,X_d]]$ (it is generated by $a_1,\dots,a_n,X_1,\dots,X_d$ over its $p^\sub{th}$-powers).

(iv): Suppose $B$ is the Henselisation of $A$ along an ideal $I\subseteq A$; so, by definition, $B$ is the filtered inductive limit of all \'etale $A$-algebras $A'$ for which the canonical map $A/I\to A'/IA'$ is an isomorphism. For each such $A'$, the diagram
\[\xymatrix{
A\ar[r]\ar[d]_F & A'\ar[d]^F\\
A\ar[r] & A'
}\]
is cocartesian by the proof of \cite[Lem.~A.9]{LangerZink2004} (with notation $R=A/pA$, $S=A'/pA'$), where $F$ denotes the absolute Frobenius homomorphism $x\mapsto x^p$. Taking the colimit over $A'$ we deduce that the diagram
\[\xymatrix{
A\ar[r]\ar[d]_F & B\ar[d]^F\\
A\ar[r] & B
}\]
is cocartesian. Since the left vertical arrow is a finite morphism by assumption, so is the right vertical arrow, as required.

(v): The proof is the same as case (iv).
\end{proof}

Next we note that properties such as F-finiteness extend from a ring to its Witt vectors:

\begin{lemma}\label{lemma_F-finiteness_of_Witt_ring}
Let $A$ be a $\bb Z_{(p)}$-algebra and let $r\ge1$. Then:
\begin{enumerate}
\item $W_r(A)$ is a $\bb Z_{(p)}$-algebra.
\item If $p$ is nilpotent in $A$ then $p$ is nilpotent in $W_r(A)$.
\item If $A$ is F-finite then $W_r(A)$ is also F-finite.
\end{enumerate}
\end{lemma}
\begin{proof}

(i): If an integer is invertible in $A$ then it is also invertible in $\bb W_S(A)$ for any truncation set $S$; see e.g., \cite[Lem.~1.9]{Hesselholt2010}.

(ii): It is well-known that $W_r(\bb F_p)=\bb Z/p^r\bb Z$, whence $p^r=0$ in $W_r(A/pA)$; since the kernel of $W_r(A)\to W_r(A/pA)$ is nilpotent by Lemma \ref{lemma_witt_1}(ii) and the assumption, it follows that $p$ is also nilpotent in $W_r(A)$.

(iii): The formula $FV=p$, e.g.~\cite[A.4(vi)]{Rulling2007}, implies that the Frobenius $F$ induces a ring homomorphism \[A\cong W_{r+1}(A)/VW_r(A)\xto{F}W_r(A)/pW_r(A),\] which is a finite morphism by Langer-Zink; since $A$ is F-finite by assumption, it follows from Lemma \ref{lemma_conditions_for_F-finite}(i) that $W_r(A)/pW_r(A)$ is also F-finite.
\end{proof}

Finally in this section on F-finiteness, we observe that F-finiteness is a sufficient (and, in fact, necessary) condition for the finite generation of K\"ahler differentials; this fact underlies the finite generation of Andr\'e--Quillen homology which we will prove in Theorem \ref{theorem_F_finite_implies_AQ_finite}:

\begin{lemma}\label{lemma_finite_generation_of_differentials}
Let $A$ be a Noetherian, F-finite $\bb Z_{(p)}$-algebra. Then, for all $n\ge0$:
\begin{enumerate}
\item $\Omega_A^n\otimes_AA/pA$ is a finitely generated $A$-module.
\item If $p$ is nilpotent in $A$, then $\Omega_A^n$ is a finitely generated $A$-module.
\end{enumerate}
\end{lemma}
\begin{proof}
It is enough to treat the case $n=1$. Let $a_1,\dots,a_m$ generate $A/pA$ as a module over its $p^\sub{th}$-powers. Then any $b\in A$ may be written as a sum $b=pb'+\sum_ib_i^pa_i$ for some $b',b_1,\dots,b_m\in A$, and so we deduce that \[db=p\,db'+\sum_i(b_i^p\,da_i+pb_i^{p-1}a_i\,db_i)\equiv\sum_ib_i^p\,da_i\mod{p}\] That is, $\Omega_{A/k}^1\otimes_AA/pA$ is generated by $da_1,\dots,da_m$.

Multiplying by $p^e$, it follows that $p^e\Omega_A^1/p^{e+1}\Omega_A^1$ is also finitely generated; so if $p$ is nilpotent then $p^e\Omega_A^1$, $e\ge1$, defines a finite filtration on $\Omega_A^1$ with finitely generated steps, whence $\Omega_A^1$ is finitely generated.
\end{proof}

\section{Finite generation results for $\HH$, $\THH$, and $\TR^r$}\label{section_finite_generation}
The primary aim of this section is Theorem \ref{theorem_finite_gen_of_THH_TR}, which states that the algebraic and topological Hochschild homology groups, with finite coefficients, of a Noetherian, F-finite $\bb Z{(p)}$-algebra $A$ are finitely generated over $A$, and that the groups $\TR^r_n(A;\bb Z/p^v)$ are finitely generated over $W_r(A)$. From this we deduce additional finite generation results for the $p$-completed theories, and for rings in which $p$ is nilpotent. We believe these are the first general finite generation results for topological Hochschild homology.

The key step is similar finite generation results for the Andr\'e--Quillen homology of $A$; indeed, the following lemma reduces the problem to Andr\'e--Quillen homology:

\begin{lemma}\label{lemma_finite_generation_reduction}
Let $A$ be a Noetherian ring, and $I\subseteq A$ an ideal. Consider the following statements:
\begin{enumerate}
\item[(i)] $D_n^i(A/\bb Z, A/I)$ is a finitely generated $A$-module for all $n,i\ge0$.
\item[(i')] $D_n^i(A/\bb Z, M)$ is a finitely generated $A$-module for all $n,i\ge0$ and all finitely generated $A$-modules $M$ killed by a power of $I$.
\item[(ii)] $\HH_n(A, A/I)$ is a finitely generated $A$-module for all $n\ge0$.
\item[(ii')] $\HH_n(A, M)$ is a finitely generated $A$-module for all $n\ge0$ and all finitely generated $A$-modules $M$ killed by a power of $I$.
\item[(iii)] $\THH_n(A, A/I)$ is a finitely generated $A$-module for all $n\ge0$.
\item[(iii')] $\THH_n(A, M)$ is a finitely generated $A$-module for all $n\ge0$ and all finitely generated $A$-modules $M$ killed by a power of $I$.
\end{enumerate}
Then (i)$\iff$(i')$\implies$(ii)$\iff$(ii')$\iff$(iii)$\iff$(iii').
\end{lemma}
\begin{proof}
The implications (i)$\Rightarrow$(i'), (ii)$\Rightarrow$(ii'), and (iii)$\Rightarrow$(iii') (whose converses are trivial) follow in the usual way using universal coefficient spectral sequences and long exact homology sequences; we will demonstrate the argument in the case of Andr\'e--Quillen homology. Let $M$ be a finitely generated $A$-module killed by a power of $I$; by induction on the power of $I$ killing $M$, and using the long exact sequence \[\cdots\To D_n^i(A/\bb Z,IM)\To D_n^i(A/\bb Z,M)\To D_n^i(A/\bb Z,M/IM)\To\cdots,\] we may clearly assume that $IM=0$. Thus $M\otimes_AA/I=M$ and so there is a universal coefficient spectral sequence \[E_{st}^2=\Tor_s^A(M,D_t^i(A/\bb Z,A/I))\Longrightarrow D_{s+t}^i(A/\bb Z,M).\] Assuming (i), the $A$-modules on the left are finitely generated, hence the abutment is also finitely generated, proving (i').

Implication (i)$\Rightarrow$(ii) is an immediate consequence of the Andr\'e--Quillen-to-Hochschild-homology spectral sequence \[E^2_{ij}=D_i^j(A/\bb Z,A/I)\Longrightarrow \HH_{i+j}(A,A/I).\]
The implication (ii')$\Rightarrow$(iii) follows from the results recalled in Section \ref{subsection_THH}. Indeed, the Pirashvili--Waldhausen spectral sequence \[E^2_{ij}=\HH_i(A,\THH_j(\bb Z,A/I))\implies \THH_{i+j}(A,A/I)\] will prove this implication if we know that $\THH_j(\bb Z,A/I)$ is a finitely generated $A$-module for all $j$; fortunately B\"okstedt's calculation shows that this is indeed the case:
\[\THH_j(\bb Z,A/I)\cong\begin{cases}A/I&j=0\\A/(I+mA)&j=2m-1\\(A/I)[m]&j=2m>0\end{cases}\]

It remains to prove (iii')$\Rightarrow$(ii'), although we will not need this implication in the remainder of the paper. To proceed by induction, fix $n\ge0$ and suppose we have shown that $\HH_i(A,M)$ is finitely generated for all $i<n$ and for all finitely generated $A$-modules $M$ killed by a power of $I$; note that the base case of the induction is covered by the identity $\HH_0(A,M)=M$. To prove the desired implication, it is now enough to fix an $A$-module $M$ killed by a power of $I$ and to show that $\HH_n(A,M)$ is finitely generated. The inductive hypothesis and B\"okstedt's calculation shows that, in the Pirashvili--Waldhausen spectral sequence \[E^2_{ij}=\HH_i(A,\THH_j(\bb Z,M))\implies \THH_{i+j}(A,M),\] the $A$-modules $E^2_{ij}$ are finitely generated for all $i<n$. Since $\THH_n(A,M)$ is finitely generated by assumption, it easily follows, e.g.~by working in the Serre quotient category of $A$-modules modulo the finitely generated modules, that $E_{n\,0}^2=\HH_n(A,M)$ is also finitely generated, as required.
\end{proof}

To prove the fundamental results about finite generation of Andr\'e--Quillen homology we will systematically work with simplicial modules and their homotopy theory, albeit not in a deep way; this could be avoided by making use of various spectral sequences, but at the expense of considerably lengthening the proofs and essentially repeating the same arguments. Therefore we now introduce some notation and make a few observations which will be repeatedly used in what follows.

Given a commutative ring $A$, let $sA\op{-mod}$ denote the category of simplicial $A$-modules, which is equivalent via the Dold--Kan correspondence to the category of chain complexes vanishing in negative degrees. We will be particularly interested in those simplicial $A$-modules having finitely generated homotopy over $A$; i.e.,  those $M_\blob\in sA\op{-mod}$ for which $\pi_n(M_\blob)$ is a finitely generated $A$-module for all $n\ge0$.

The following permanence properties of such simplicial $A$-modules will be essential:

\begin{lemma}\label{lemma_fg_simplicial}
Let $A$ be a Noetherian ring. Then:
\begin{enumerate}
\item If $L_\blob\to M_\blob\to N_\blob$ is a homotopy cofibre sequence of simplicial $A$-modules such that two of $L_\blob$, $M_\blob$, $N_\blob$ have finitely generated homotopy over $A$, then so does the third.
\item If $L_\blob\in sA\op{-mod}$ consists of flat $A$-modules in each degree and has finitely generated homotopy, and $B$ is a Noetherian $A$-algebra, then $L_\blob\otimes_AB$ has finitely generated homotopy over $B$.
\item If $L_\blob, M_\blob\in sA\op{-mod}$ have finitely generated homotopy, and $L_\blob$ consists of flat $A$-modules in each degree, then $L_\blob\otimes_AM_\blob$ has finitely generated homotopy.
\end{enumerate}
\end{lemma}
\begin{proof}
(i) is an obvious consequence of the long exact sequence of $A$-modules \[\cdots\To\pi_n(L_\blob)\To\pi_n(M_\blob)\To\pi_n(N_\blob)\To\cdots\]

(iii): The flatness assumption implies that there is a spectral sequence of $B$-modules $E^2_{ij}=\op{Tor}_i^A(\pi_j(L_\blob),B)\implies \pi_{i+j}(L_\blob\otimes_AB)$, so it is enough to show that the Tors are finitely generated $B$-modules; since $\pi_j(L_\blob)$ is finitely generated over $A$ by assumption, we may pick a resolution $P_\blob$ of it by finitely generated, projective $A$-modules, and then $\op{Tor}_i^A(\pi_j(L_\blob),B)=H_i(P_\blob\otimes_AB)$ is evidently finitely generated over $B$.

(iv): The flatness assumption implies that there is a spectral sequence of $A$-modules $E_{ij}^1=L_p\otimes_A\pi_i(M_\blob)\implies \pi_{i+j}(L_\blob\otimes_AM_\blob)$, whose $E_2$-page is $E_2^{ij}=\pi_i(L_\blob\otimes_A\pi_j(M_\blob))$, where. But these terms on the $E_2$-page are finitely generated thanks to our assumption and the spectral sequences $^\prime E_2^{st}=\op{Tor}^A_s(\pi_t(L_\blob),\pi_j(M_\blob))\implies\pi_{s+t}(L_\blob\otimes_A\pi_j(M_\blob))$ for each $j\ge0$.
\end{proof}

The second tool to prove the required results about finite generation of Andr\'e--Quillen homology is a spectral sequence due originally to C.~Kassel \& A.~Sletsj\o e  \cite[Thm.~3.2]{Kassel1992}, which was rediscovered in \cite{Krishna2010} and \cite{Morrow_birelative}. We state it here as a filtration on the cotangent complexes, rather than as the resulting spectral sequence:

\begin{lemma}[Kassel--Sletsj\o e]\label{lemma_Kassel_Sletjoe}
Let $A\to B\to C$ be homomorphisms of rings. Then it is possible to choose the cotangent complexes $\bb L_{C/A}$, $\bb L_{C/B}$, and $\bb L_{B/A}$ to be degree-wise projective modules and, for all $i\ge0$, so that $\bb L_{C/A}^i$ has a natural filtration \[\bb L_{C/A}^i=F^0\bb L_{C/A}^i\supseteq F^1\bb L_{C/A}^i\supseteq\cdots\supseteq F^i\bb L_{C/A}^i=\bb L_{B/A}^i\otimes_BC\supseteq F^{i+1}\bb L_{C/A}^i=0\] of length $i$ by simplicial $C$-modules, with graded pieces \[\op{gr}^j\bb L_{C/A}^i\cong (\bb L_{B/A}^j\otimes_BC)\otimes_C\bb L_{C/B}^{i-j}\] for $j=0,\dots,i$.
\end{lemma}
\begin{proof}[Sketch of proof]
One chooses simplicial resolutions in the usual way \cite[Thm.~5.1]{Quillen1970} to ensure that the Jacobi--Zariski sequence \[0\to \bb L_{B/A}\otimes_BC\to\bb L_{C/A}\to\bb L_{C/B}\to 0\] is actually a short exact sequence of simplicial $C$-modules which are projective in each degree. Then observe that whenever $0\to L\to M\to N\to 0$ is a short exact sequence of projective modules over $C$, there is a resulting filtration on $\bigwedge^iM$ with graded pieces $\bigwedge^jL\otimes_C\bigwedge^{i-j}N$.
\end{proof}

With these tools at hand, we may now begin to prove the main results about finite generation of Andr\'e--Quillen homology; to simplify our statements, it is convenient to make the following definition:

\begin{definition}
Let $m\ge0$ be an integer. We say that a homomorphism $k\to A$ of Noetherian rings is {\em $m$-AQ-finite}, or that $A$ is $m$-AQ-finite over $k$, if and only if the Andr\'e--Quillen homology groups $D_n^i(A/k,A/mA)$ are finitely generated $A$-modules for all $n,i\ge0$.

In the case $m=0$ we omit the ``$0$-''; so $A$ is AQ-finite over $k$ if and only if $D_n^i(A/k)$ is finitely generated for all $n,i\ge0$.

In the case $k=\bb Z$ we omit the ``over $\bb Z$'', and say simply that $A$ is $m$-AQ-finite or AQ-finite.
\end{definition}

\begin{lemma}\label{lemma_finite_gen_of_AQ}
Assume all rings in the following assertions are Noetherian, and let $m\ge0$. Then:
\begin{enumerate}
\item A finite type morphism is $m$-AQ-finite.
\item Localising at a multiplicative system is $m$-AQ-finite.
\item A composition of two $m$-AQ-finite morphisms is again $m$-AQ-finite.
\end{enumerate}
\end{lemma}
\begin{proof}
Since AQ-finiteness implies $m$-AQ-finiteness (this follows from the implication (i)$\Rightarrow$(i') of Lemma \ref{lemma_finite_generation_reduction}), it is enough to prove (i) and (ii) in the case $m=0$.

Then (i) is a result of Quillen, obtained by constructing a simplicial resolution of the finitely generated $k$-algebra $A$ by finitely generated, free $k$-algebras; see \cite[Prop.~4.12]{Quillen1970}. Claim (ii) is another result of Quillen: if $S$ is a multiplicative system in $k$, then \cite[Thm.~5.4]{Quillen1970} states that $D_n^i(S^{-1}k/k)=0$.

(iii): Let $A\to B\to C$ be homomorphisms of Noetherian rings such that $A\to B$ and $B\to C$ are $m$-AQ-finite. Pick the cotangent complexes according to Lemma \ref{lemma_Kassel_Sletjoe}, and fix $i\ge0$. Since all the simplicial modules appearing in the statement of Lemma \ref{lemma_Kassel_Sletjoe} are degree-wise projective, the description of the filtration remains valid after tensoring by any $C$-module. In particular, we deduce that $\bb L_{C/A}^i\otimes_CC/mC$ has a decreasing filtration with graded pieces.
\begin{align*}
\op{gr}^j(\bb L_{C/A}^i\otimes_CC/mC)
	&\cong (\bb L_{B/A}^j\otimes_BC)\otimes_C\bb L_{C/B}^{i-j}\otimes_CC/mC\\
	&\cong ((\bb L_{B/A}^j\otimes_BB/mB)\otimes_{B/mB}C/mC)\otimes_{C/mC}(\bb L_{C/B}^{i-j}\otimes_CC/mC)
\end{align*} for $j=0,\dots,i$. By the $m$-AQ-finiteness assumption, $\bb L^j_{B/A}\otimes_BB/mB$ and $\bb L_{C/B}^{i-j}\otimes_CC/mC$ and have finitely generated homotopy over $B$ and $C$, respectively; so Lemma \ref{lemma_fg_simplicial}(iii)+(iv) imply that the above graded pieces have finitely generated homotopy. Applying Lemma \ref{lemma_fg_simplicial}(i) $i$-times, it follows that $\bb L^i_{C/A}\otimes_CC/mC$ has finitely generated homotopy, as desired.
\end{proof}

The following finite generation result is the first main result of the paper; it states, in particular, that if $A$ is a Noetherian, F-finite $\bb F_p$-algebra, then the Andr\'e--Quillen homology groups $D_n^i(A/\bb F_p)$ are finitely generated $A$-modules for all $n,i\ge0$:

\begin{theorem}\label{theorem_F_finite_implies_AQ_finite}
Let $p$ be a prime number.
\begin{enumerate}
\item Let $e\ge 1$; then any Noetherian, F-finite $\bb Z/p^e\bb Z$-algebra is AQ-finite over $\bb Z/p^e\bb Z$ and over $\bb Z$.
\item Any Noetherian, F-finite $\bb Z_{(p)}$-algebra is $p$-AQ-finite over $\bb Z_{(p)}$ and over $\bb Z$.
\end{enumerate}
\end{theorem}
\begin{proof}
First let $A$ be a Noetherian, F-finite $\bb F_p$-algebra. Let $F:A\to A$ be the absolute Frobenius $x\mapsto x^p$. The homomorphism $F$ may be factored as a composition of $\bb F_p$-algebra homomorphisms \[A\xTo{\pi}A^p\xTo{e}A,\] where $A^p$ is the subring of $A$ consisting of $p^\sub{th}$-powers, $\pi(x):=x^p$, and $e$ is the natural inclusion. Our proof will be based on the following observations:
\begin{enumerate}
\item $\pi$ and $e$ are finite type homomorphism between Noetherian rings (indeed, $A^p$ is a quotient of the Noetherian ring $A$, and the fact that $e$ is of finite type is exactly our assumption that $A$ is F-finite).
\item The absolute Frobenius $F:\bb L_{A/\bb F_p}^i\to\bb L_{A/\bb F_p}^i$ is zero for $i\ge1$ (indeed, in any $\bb F_p$-algebra, we have $F(d\beta)=d\beta^p=p\beta^{p-1}\,d\beta=0$).
\end{enumerate}
We will now prove the following by induction on $i\ge0$: \begin{quote}If $A$ is a Noetherian, F-finite $\bb F_p$-algebra, then $\bb L_{A/\bb F_p}^i$ has finitely generated homotopy over $A$.\end{quote} The claim is trivial for $i=0$ since $\bb L_{A/\bb F_p}^0\simeq A$, so assume $i\ge1$. We remark that, in the inductive proof that follows, we will need to choose different models of the cotangent complex $\bb L_{A/\bb F_p}$

Apply Lemma \ref{lemma_Kassel_Sletjoe} to the composition $\bb F_p\to A^p\xto{e} A$ to see that it is possible to pick the cotangent complexes $\bb L_{A/\bb F_p}$, $\bb L_{A/A^p}$, and $\bb L_{A^p/A}$ in such a way that $\bb L_{A/\bb F_p}^i$ has a descending filtration $F^\blob \bb L_{A/\bb F_p}^i$ with graded pieces \[\op{gr}^j\bb L_{A/\bb F_p}^i\cong (\bb L_{A^p/\bb F_p}^j\otimes_{A^p}A)\otimes_A\bb L_{A/A^p}^{i-j}\tag{\dag}\] for $j=0,\dots,i$. For each $j=0,\dots,i-1$, the simplicial $A^p$-module $\bb L_{A^p/\bb F_p}^j$ has finitely generated homotopy over $A^p$ by the inductive hypothesis (note that $A^p$ is also a Noetherian, F-finite $\bb F_p$-algebra since it is a quotient of $A$); meanwhile, $\bb L_{A/A^p}^{i-j}$ has finitely generated homotopy over $A$ by Lemma \ref{lemma_finite_gen_of_AQ}(i). Applying Lemma \ref{lemma_fg_simplicial}(iii)+(iv) we deduce that the right side of (\dag) has finitely generated homotopy over $A$ for $j=0,\dots,i-1$; from Lemma \ref{lemma_fg_simplicial}(i) it follows that $X:=\bb L^i_{A/\bb F_p}/F^i\bb L^i_{A/\bb F_p}$ has finitely generated homotopy. In summary, we have a homotopy cofibre sequence \[\op{gr}^i\bb L_{A/\bb F_p}^i=\bb L_{A^p/\bb F_p}^i\otimes_{A^p}A\xTo{e} \bb L_{A/\bb F_p}^i\To X,\] where $X$ has finitely generated homotopy over $A$.

We now apply a similar argument to the composition $\bb F_p\to A\xto{\pi}A^p$, after first picking new models for the cotangent complexes $\bb L_{A^p/\bb F_p}$, $\bb L_{A^p/A}$, and $\bb L_{A/\bb F_p}$ such that $\bb L_{A^p/\bb F_p}^i$ has a descending filtration $F^\blob \bb L_{A^p/\bb F_p}^i$ with graded pieces \[\op{gr}^j\bb L_{A^p/\bb F_p}^i\cong (\bb L_{A/\bb F_p}^j\otimes_AA^p)\otimes_{A^p}\bb L_{A^p/A}^{i-j}\tag{\dag}\] for $j=0,\dots,i$. The inductive hypothesis implies that $\bb L_{A/\bb F_p}^j$ has finitely generated homotopy over $A$ for $j=0,\dots,i-1$, and Lemma \ref{lemma_finite_gen_of_AQ}(i) implies that $\bb L_{A^p/A}^{i-j}$ has finitely generated homotopy over $A^p$. So by the same argument as in the previous paragraph, $\bb L^i_{A^p/\bb F_p}/F^i\bb L^i_{A^p/\bb F_p}$ has finitely generated homotopy over $A^p$. Using Lemma \ref{lemma_fg_simplicial}(iii) to base change along $A^p\xto{e} A$, we deduce that there is a homotopy cofibre sequence \[\bb L_{A/\bb F_p}^i\otimes_AA^p\otimes_{A^p}A\xTo{\pi} \bb L_{A^p/\bb F_p}^i\otimes_{A^p}A\To Y\] of simplicial $A$-modules, where $Y$ has finitely generated homotopy over $A$.

In conclusion, we have a diagram in the homotopy category of simplicial $A$-modules
\[\xymatrix@=6mm{
Y & & \\
\bb L^i_{A^p/\bb F_p}\otimes_{A^p}A \ar[r]^-{e} \ar[u] & \bb L^i_{A/\bb F_p} \ar[r] & X \\
\bb L^i_{A/\bb F_p}\otimes_AA^p\otimes_{A^p}A \ar[u]^\pi &&
}\]
in which the row and column are both homotopy cofibre sequences, and in which $X$ and $Y$ have finitely generated homotopy. By commutativity of taking homotopy colimits, there is therefore a resulting homotopy cofibre sequence \[Y\To\op{hocofib}(\bb L_{A/\bb F_p}^i\otimes_AA^p\otimes_{A^p}A\xTo{e\circ\pi}\bb L_{A/\bb F_p}^i)\To X,\] and so Lemma \ref{lemma_fg_simplicial}(i) implies that the simplicial $A$-module in the centre of this sequence also has finitely generated homotopy. But by observation (ii) earlier in the proof, $e\circ\pi=F$ is nulhomotopic on the cotangent complexes, so this means that both $\bb L_{A/\bb F_p}^i\otimes_AA^p\otimes_{A^p}A$ and $\bb L_{A/\bb F_p}^i$ must have finitely generated over homotopy over $A$, which completes the proof of the inductive step.

We have proved that any Noetherian, F-finite $\bb F_p$-algebra is AQ-finite over $\bb F_p$; the remaining claims of the theorem will all follow from this. Firstly let $k$ denote either $\bb Z/p^e\bb Z$ or $\bb Z_{(p)}$, and note that any Noetherian, F-finite $\bb F_p$-algebra is also AQ-finite over $k$; this follows from Lemma \ref{lemma_finite_gen_of_AQ}. 

Now let $A$ be a Noetherian, F-finite $k$-algebra. We will prove by induction on $i\ge0$ that $\bb L_{A/k}^i\otimes_AA/pA$ has finitely generated homotopy over $A/pA$. We apply Lemma \ref{lemma_Kassel_Sletjoe} to the composition $k\to A\to A/pA$ and note the following, similar to the earlier part of the proof: $\bb L_{(A/pA)/k}^i$ has finitely generated homotopy by what we have already proved, while $\op{gr}^j\bb L_{(A/pA)/k}^i\cong (\bb L_{A/k}^j\otimes_AA/pA)\otimes_{A/pA}\bb L_{(A/pA)/A}^{i-j}$ has finitely generated homotopy over $A/pA$ for $j=0,\dots,i-1$ by Lemma \ref{lemma_fg_simplicial} and the inductive hypothesis. By Lemma \ref{lemma_fg_simplicial} again, it follows that the remaining part of the filtration, namely $\op{gr}^i\bb L_{(A/pA)/k}^i\cong \bb L_{A/k}^i\otimes_AA/pA$, also has finitely generated homotopy over $A/pA$, as required.

We have proved that any Noetherian, F-finite $k$-algebra is $p$-AQ-finite over $k$, hence also over $\bb Z$ by Lemma \ref{lemma_finite_gen_of_AQ}. This completes the proof of (ii).

From now on $k=\bb Z/p^e\bb Z$ for some $e\ge1$. All that remains to be shown is that if a Noetherian $k$-algebra $A$ is $p$-AQ-finite over $k$, then it is actually AQ-finite over $k$. By implication (i)$\Rightarrow$(ii) of Lemma \ref{lemma_finite_generation_reduction} with $I=pA$, it follows that $D_n^i(A/k,M)$ is finitely generated for any finitely generated $A$-module $M$ which is killed by a power of $p$; in particular, taking $M=A=A/p^eA$ we deduce that $D_n^i(A/k)$ is finitely generated for all $n,i\ge0$, as required.
 \end{proof}

The following is our main finite generation theorem for $\HH$, $\THH$, and $\TR^r$, from which others will follow:

\begin{theorem}\label{theorem_finite_gen_of_THH_TR}
Let $A$ be a Noetherian, F-finite $\bb Z_{(p)}$-algebra, and $v$ a positive integer. Then:
\begin{enumerate}
\item $\HH_n(A;\bb Z/p^v)$ and $\THH_n(A;\bb Z/p^v)$ are finitely generated $A$-modules for all $n\ge~0$.
\item $\TR_n^r(A;\bb Z/p^v)$ is a finitely generated $W_r(A)$-module for all $n\ge 0$, $r\ge1$.
\end{enumerate}
\end{theorem}
\begin{proof}
\comment{
(i): There are short exact sequences $0\to A[p^v]\to A\xto{\pi} p^vA\to 0$ and $0\to p^vA\xto{i} A\to A/p^vA\to 0$, where the map $i\pi:A\to A$ is multiplication by $p^v$. Taking the associated long exact sequences of Hochschild homology, and noting that $\HH_n(A,A[p^v])$ and $\HH_n(A,A/p^vA)$ are finitely generated $A$-modules by Theorem \ref{theorem_F_finite_implies_AQ_finite} and Lemma \ref{lemma_finite_generation_reduction}, it easily follows that the maps \[\pi:\HH_n(A)\to \HH_n(A,p^vA),\quad i:\HH_n(A,p^vA)\to HH_n(A)\] have finitely generated kernel and cokernel; hence the same is true of $i\pi=p^v$, which exactly means that $\HH_n(A)$ has finitely generated $p^v$-torsion and $p^v$-cotorsion for all $n\ge0$. Thanks to the standard short exact sequence of $A$-modules \[0\To \HH_n(A)/p^v\HH_n(A)\To\HH_n(A;\bb Z/p^v)\To\HH_{n-1}(A)[p^v]\To 0,\] this completes the proof of the finite generation claim for $\HH$. The argument for $\THH$ is verbatim equivalent.
}
(i): As explained in Section \ref{subsection_p_completions}, there is a long exact sequence of $A$-modules \[\cdots\To\HH_{n-1}(A,A[p^v])\To\HH_n(A;\bb Z/p^v)\To\HH_n(A,A/p^vA)\To\cdots\] The $A$-modules $\HH_n(A,A[p^v])$ and $\HH_n(A,A/p^vA)$ are finitely generated for all $n\ge0$ by Theorem \ref{theorem_F_finite_implies_AQ_finite} and Lemma \ref{lemma_finite_generation_reduction}, whence $\HH_n(A;\bb Z/p^v)$ is also finitely generated. The argument for $\THH$ is verbatim equivalent.

(ii): The fundamental long exact sequence \[\cdots\To\pi_n(\THH(A;\bb Z/p^v)_{hC_{p^r}})\To \TR^{r+1}_n(A;\bb Z/p^v)\To \TR^r_n(A;\bb Z/p^v)\to\cdots\] is one of $W_{r+1}(A)$-modules, as explained in Sections \ref{subsection_Witt_structure} and \ref{subsection_p_completions}. Since the ring $W_{r+1}(A)$ is Noetherian by Langer--Zink (Theorem \ref{theorem_Langer_Zink}), it is enough by the five lemma and induction (recall that $\TR_n^1(A;\bb Z/p^v)=\THH_n(A;\bb Z/p^v)$ to start the induction) to show that $\pi_n(\THH(A;\bb Z/p^v)_{hC_{p^r}})$ is a finitely generated $W_{r+1}(A)$-module for all $n,r\ge0$.

To show this, recall the group homology spectral sequence \[E^2_{ij}=H_i(C_{p^r},\THH_j(A;\bb Z/p^v))\implies \pi_{i+j}(\THH(A;\bb Z/p^v)_{hC_{p^r}}),\] which is a spectral sequence of $W_{r+1}(A)$-modules, where $W_{r+1}(A)$ acts on the $A$-modules on the $E^2$-page via $F^r:W_{r+1}(A)\to A$. But each $A$-module $E^2_{st}$ is finitely generated by part (i), hence is also finitely generated as a $W_{r+1}(A)$-module since $F^r$ is a finite morphism (again by Langer--Zink). Thus the abutment of the spectral sequence is also finitely generated over $W_{r+1}(A)$, as required.
\end{proof}

Standard techniques allow us to let $v\to\infty$ in Theorem \ref{theorem_finite_gen_of_THH_TR}, thereby proving Theorem \ref{theorem_intro_finite_gen_of_THH_TR} of the Introduction:

\begin{corollary}\label{corollary_finite_gen_of_THH_TR_p_complete}
Let $A$ be a Noetherian, F-finite $\bb Z_{(p)}$-algebra. Then:
\begin{enumerate}
\item $\HH_n(A;\bb Z_p)$ and $\THH_n(A;\bb Z_p)$ are finitely generated $A_p^\comp$-modules for all $n\ge0$.
\item $\TR_n^r(A;p,\bb Z_p)$ is a finitely generated $W_r(A_p^\comp)$-module for all $n\ge 0$, $r\ge1$.
\end{enumerate}
\end{corollary}
\begin{proof}
This follows from Theorem \ref{theorem_finite_gen_of_THH_TR} via Proposition \ref{proposition_to_prove_finite_gen_p_adically} in the appendix; note that $W_r(A)_p^\comp\cong W_r(A_p^\comp)$ by Lemma \ref{lemma_W_r_of_p_completion}.
\end{proof}

In the case in which $p$ is nilpotent in $A$, for example when $A$ is an $\bb F_p$-algebra, it is evidently not necessary to $p$-complete:

\begin{corollary}\label{corollary_finite_gen_nilp}
Let $A$ be a Noetherian, F-finite $\bb Z_{(p)}$-algebra in which $p$ is nilpotent. Then:
\begin{enumerate}
\item $\HH_n(A)$ and $\THH_n(A)$ are finitely generated $A$-modules for all $n\ge0$.
\item $\TR_n^r(A;p)$ is a finitely generated $W_r(A)$-module for all $n\ge 0$, $r\ge1$.
\end{enumerate}
\end{corollary}
\begin{proof}
Since $p$ is nilpotent in $A$, it is also nilpotent in $W_r(A)$ by Lemma \ref{lemma_F-finiteness_of_Witt_ring}(ii); hence $\HH_n(A)$, $\THH_n(A)$, and $\TR_n^r(A;p)$ are all groups of bounded $p$-torsion. It follows that $\HH_n(A)=\HH_n(A;\bb Z_p)$, similarly for $\THH$ and $\TR$, and that $A_p^\comp=A$. So the claim follows from Corollary \ref{corollary_finite_gen_of_THH_TR_p_complete} (or it can be deduced directly from Theorem \ref{theorem_finite_gen_of_THH_TR} without passing via the $p$-completion).
\end{proof}

\begin{remark}
The assertions of Corollary \ref{corollary_finite_gen_nilp} are true whenever $A$ is a Noetherian, F-finite, AQ-finite $\bb Z_{(p)}$-algebra (e.g., an essentially finite type $\bb Z_{(p)}$-algebra).

Firstly, (i) follows immediately from Lemmas \ref{lemma_finite_gen_of_AQ} and \ref{lemma_finite_generation_reduction}. To prove (ii), one then repeats the proof of Theorem \ref{theorem_finite_gen_of_THH_TR}(ii), except working directly with $\THH(A)$ and $\TR^r(A;p)$ instead of $\THH(A;\bb Z/p^v)$ and $\TR^r(A;\bb Z/p^v)$.
\end{remark}

\section{Continuity and the pro HKR theorems}\label{section_continuity_and_HKR}
Let $A$ be a Noetherian ring and $I\subseteq A$ an ideal. Then there is a canonical map \[\THH_n(A)\otimes_AA/I\to \THH_n(A/I),\] which is rarely an isomorphism. Replacing $I$ by powers of itself obtains a morphism of pro $A$-modules, \[\{\THH_n(A)\otimes_AA/I^s\}_s\To\{\THH_n(A/I^s)\}_s,\] and the question was raised in \cite{Carlsson2001} as to when it is an isomorphism. For example, it was shown by Geisser and Hesselholt \cite[\S1]{GeisserHesselholt2006b} to be an isomorphism in the case that $A=R[X_1,\dots,X_d]$ is a polynomial algebra over any ring $R$, and $I=\pid{X_1,\dots,X_d}$. Under our usual hypotheses including F-finiteness, we will show in Section \ref{subsection_continuity} that it is an isomorphism for {\em any} ideal $I$, at least working with Hochschild homology with finite coefficients.

From such ``degree-wise continuity'' results we obtain further isomorphisms such as $\THH_n(A;\bb Z/p^v)\otimes_A\hat A\cong\THH_n(A;\bb Z/p^v)$, where $\hat A$ is the $I$-adic completion of $A$, and weak equivalences of spectra such as $\THH(A)_p^\comp\quis\holim_s\THH(A/I^s)_p^\comp$. This latter result is known as the continuity of topological Hochschild homology with respect to $I\subseteq A$.

Similar results are given for $\HH$, $\TR^r$, $\TC$, etc., and in the following special cases: firstly when $p$ is nilpotent, and secondly when $I=pA$.

\subsection{The restriction spectral sequence}
A key tool in our forthcoming proof of continuity properties is a certain restriction spectral sequence for derived Hochschild homology. This was established for topological Hochschild homology by M.~Brun \cite[Thm.~6.2.10]{Brun2000}, and the analogue in the algebraic case likely follows from his proof; however, for convenience we offer a purely algebraic proof via simplicial methods.

We remark that this is the one place in the paper where the rings need not be commutative; indeed, even if the reader restricts to commutative rings, he/she should be aware that the $B$-bimodule $\op{Tor}_q^A(B,M)$ appearing in the following result is not in general a symmetric module (though in our main case of interest, when $A$ is commutative and $B=A/I^s$, it is symmetric).

\begin{proposition}\label{proposition_Bjorns_SS}
Let $k$ be a commutative ring, $A\to B$ a morphism of $k$-algebras, and $M$ a $B$-bimodule. Then there is a natural spectral sequence of $k$-modules \[E^2_{ij}=\HH_i^k(B,\Tor_j^A(B,M))\Longrightarrow \HH_{i+j}^k(A,M).\]
\end{proposition}

Before we can give the proof of the proposition, we must recall some results about simplicial resolutions and bar complexes.

Given a commutative ring $k$, we will work with simplicial $k$-algebras, the category of which is equipped with Quillen's usual model structure \cite{Quillen1967} for categories of simplicial objects.\comment{; in particular, if $A\to B$ is a homomorphism of commutative $k$-algebras, then a cofibrant replacement of $B$ as an $A$-algebra is a simplicial resolution $A[X_\bullet]\to B$ by a free simplicial $A$-algebra $A[X_\bullet]$, where $X_\bullet$ is a simplicial set.} We have the following key lemma thanks to Quillen \cite[Thm.~5.1]{Quillen1970}, which amounts to the statement that the model structure on simplicial $k$-algebras is left proper:
\begin{quote}
Suppose $k\cofib P_\bullet\quis A$ is a cofibrant replacement of $k\to A$, and that $P_\bullet\cofib Q_\bullet\quis B$ is a cofibrant replacement of $P_\bullet\to B$ (so that $k\cofib Q_\bullet\quis B$ is a cofibrant replacement of $k\to B$). Then $A\to Q_\bullet\otimes_{P_\bullet}A\to B$ is a cofibrant replacement of $A\to B$.
\end{quote}

Let $k$ be a commutative ring, $A$ a $k$-algebra, $M$ a left $A$-module, and $N$ a right $A$-module. We denote by $B_\bullet^k(M,A,N)$ the two-sided bar complex: it is the simplicial $k$-module given by \[B_\bullet^k(N,A,M)=N\otimes_k\underbrace{A\otimes_k\cdots\otimes_kA}_{\sub{\sub{$\bullet$ times}}}\otimes_kM.\] Equivalently, $B_\bullet(N,A,M)=C_\bullet^k(A,M\otimes_kN)$, where $C_\bullet^k(A,-)$ denotes the simplicial Hochschild $k$-module of an $A$-bimodule.

The following are standard facts about bar complexes:
\begin{enumerate}
\item $B_\bullet^k(N,A,M)\cong N\otimes_AB_\bullet^k(A,A,A)\otimes_AM$.
\item $B_\bullet^k(A,A,A)=C_\bullet^k(A,A\otimes_kA^\sub{op})$, and this is a simplicial $A$-bimodule. More generally, if $M$ is an $A$-bimodule, then there is an isomorphism of simplicial $k$-modules: \[C_\bullet^k(A,M)\cong M\otimes_{A\otimes A^\sub{op}}B_\bullet^k(A,A,A),\quad a_0ma_{n+1}\otimes a_1\otimes\cdots\otimes a_n\leftrightarrow m\otimes a_0\otimes\cdots\otimes a_{n+1}\]
\item $B_\bullet^k(A,A,M)\to M$ is a resolution of $M$ as a left $A$-module; if $M$ is in fact an $A$-bimodule, then it is a resolution of $M$ as an $A$-bimodule.
\end{enumerate}

\begin{proof}[Proof of Proposition \ref{proposition_Bjorns_SS}]
Let simplicial resolutions be as in Quillen's lemma. Moreover, let $M_\bullet\to M$ be a projective simplicial resolution of $M$ as a left $Q_\bullet$-module; i.e., $M_\bullet\to M$ is a quasi-isomorphism of simplicial left $Q_\bullet$-modules, and $M_i$ is a projective left $Q_i$-module for each $i\ge0$.

Form the simplicial $k$-module \[X_\bullet=C_\bullet(Q_\bullet,B_\bullet(Q_\bullet,P_\bullet,M_\bullet)),\] whose homotopy groups we will calculate in two ways. Firstly, $X_\bullet$ is the diagonal of the bisimplicial $k$-module \[\langle X_{ij}\rangle_{ij}=\langle C_i(Q_j,B_j(Q_j,P_j,M_j))\rangle_{ij}.\] For each fixed $j\ge 0$, the simplicial $k$-module $C_\bullet(Q_j,B_j(Q_j,P_j,M_j))$ calculates \[\Tor_*^{Q_j\otimes_kQ_j^\sub{op}}\big(B_j(Q_j,P_j,M_j),Q_j\big).\] But $B_j(Q_j,P_j,M_j)$ is a projective $Q_j\otimes_kQ_j^\sub{op}$-module, so these $\Tor$ groups vanish in higher degree; therefore the natural map $X_\bullet\to Q_\bullet\otimes_{Q_\bullet\otimes_kQ_\bullet^\sub{op}}B_\bullet(Q_\bullet,P_\bullet,M_\bullet))$ is a quasi-isomorphism. Furthermore, for each $i$, the natural map \[M_i\to Q_i\otimes_{Q_i\otimes_kQ_i^\sub{op}}(M_i\otimes_kQ_i)\] is an isomorphism $k$-modules (this is obvious if $M_i$ is a free left $Q_i$-module, hence follows for the case of a projective module), and so we have an isomorphism of simplicial $k$-modules: \[C_\bullet(P_\bullet,M_\bullet)\isoto Q_\bullet\otimes_{Q_\bullet\otimes_kQ_\bullet^\sub{op}}B_\bullet(Q_\bullet,P_\bullet,M_\bullet)\] In conclusion, $X_\bullet\simeq C_\bullet(P_\bullet,M_\bullet)$, and $C_\bullet(P_\bullet,M_\bullet)\to C_\bullet(P_\bullet,M)$ is obviously a quasi-isomorphism; so $X_\bullet$ calculates $\HH_*^k(A,M)$.

Secondly, for each fixed $i$, the augmentation map $B_\bullet(P_i,P_i,M_i)\to M_i$ is a quasi-isomorphism and remains so after applying $C_i(Q_i,Q_i\otimes_{P_i}-)$; it easily follows that $X_\bullet$ is quasi-isomorphic to $C_\bullet(Q_\bullet,Q_\bullet\otimes_{P_\bullet}M_\bullet)=C_\bullet(Q_\bullet,R_\blob\otimes_AM_\bullet)$, where $R_\bullet:=Q_\bullet\otimes_{P_\bullet}A$. For each $i$, we may factor the natural map of simplicial $k$-algebras $R_i\to B$ as $R_i\cofib R_{i\bullet}\quis B$; since both $R_\bullet$ and the diagonal of $R_{\bullet\bullet}$ are quasi-isomorphic to $B$, we see that $X_\bullet$ is quasi-isomorphic to the diagonal of the bisimplicial $k$-algebra \[Y_{\bullet\bullet}=\langle C_i(Q_i,R_{ij}\otimes_AM_i)\rangle_{ij}\] Since $M_\bullet\quis M$ we may also replace $M_i$ by $M$ in the definition of $Y$.

For each fixed $i$, the simplicial $A$-algebra $R_{i\bullet}$ is a projective resolution of $B$, so $R_{i\bullet}\otimes_AM_i$ calculates $\op{Tor}_*^A(B,M)$. The usual spectral sequence of a bisimplicial group completes the proof.
\end{proof}

\comment{
Suppose that $H$ is a functor from commutative, Noetherian rings to abelian groups such that $H(A)$ is naturally an $A$-module; e.g., $H=\HH_n,\THH_n$. Then, given an ideal of $I$, we have a canonical map of $A/I$-modules , which need not be an isomorphism; taking the limit over powers of $I$ yields a map of pro $A$-modules  \todo{The language of flat pro $A$-algebras could vastly simplify this... Or even etale.} The theme of this section is to show that this map is an isomorphism in our cases of interest.
}

\subsection{Degree-wise continuity and continuity of $\HH$, $\THH$, $\TR^r$, etc.}\label{subsection_continuity}
Combining the restriction spectral sequence of Proposition \ref{proposition_Bjorns_SS} with our earlier finite generation results and the Artin--Rees theorem, we may now establish our ``degree-wise continuity'' results, starting with the following lemma:

\begin{lemma}\label{lemma_continuity}
Let $A$ be a Noetherian ring, $I\subseteq A$ an ideal, $M$ a finitely generated $A$-module, and $n\ge0$. Consider the canonical maps:
\begin{itemize}
\item[] $\{\HH_n(A,M)\otimes_AA/I^s\}_s\xto{(ii)}\{\HH_n(A,M/I^sM)\}_s\xto{(i)} \{\HH_n(A/I^s,M/I^sM)\}_s$
\item[] $\{\THH_n(A,M)\otimes_AA/I^s\}_s\xto{(ii)}\{\THH_n(A,M/I^sM)\}_s\xto{(i)} \{\THH_n(A/I^s,M/I^sM)\}_s$
\end{itemize}
Then maps (i) are isomorphisms. If $A$ is further assumed to be $m$-AQ-finite for some $m\ge0$, and $M$ is annihilated by $m$, then maps (ii) are also isomorphisms.
\end{lemma}
\begin{proof}
(i): By Proposition \ref{proposition_Bjorns_SS}, for each $s\ge1$ there is a first quadrant spectral sequence of $A$-modules \[E^2_{ij}(s)=\HH_i(A/I^s,\Tor_j^A(A/I^s,M/I^sM))\implies \HH_{i+j}(A,M/I^sM).\] These assemble to a spectral sequence of pro $A$-modules \[E^2_{ij}(\infty)=\{\HH_i(A/I^s,\Tor_j^A(A/I^s,M/I^sM))\}_s\implies \{\HH_{i+j}(A,M/I^sM)\}_s.\] But $\{\Tor_j^A(A/I^s,M/I^sM)\}_s=0$ for all $j\ge1$, by Corollary \ref{corollary_Artin_Rees}, so this spectral sequence collapses to edge map isomorphisms $\{\HH_n(A,M/I^sM)\}_s\isoto\{\HH_n(A/I^s,M/I^sM)\}_r$ of pro $A$-modules for all $n\ge0$; this proves isomorphism (i) for $\HH$.

Isomorphism (i) for $\THH$ can be proved in exactly the same way as $\HH$, using Brun's $\THH$ version of Proposition \ref{proposition_Bjorns_SS} in \cite[Thm.~6.2.10]{Brun2000}. However, we will show that it can also be deduced from the $\HH$ case using the Pirashvili--Waldhausen spectral sequence, which was described in Section \ref{subsection_THH}.

The P.--W.\ spectral sequences for $A$ and its modules $M/I^sM$ assemble to a spectral sequence of pro $A$-modules  \[F^2_{ij}(\infty)=\{\HH_i(A,\THH_j(\bb Z,M/I^sM))\}_s\implies \{\THH_{i+j}(A,M/I^sM)\}_s.\] Similarly, the P.--W.\ spectral sequences for the rings $A/I^s$ and modules $M/I^sM$ assemble to a spectral sequence of pro $A$-modules
\[^\prime F^2_{ij}(\infty)=\{\HH_i(A/I^s,\THH_j(\bb Z,M/I^sM))\}_s\implies \{\THH_{i+j}(A/I^sM,M/I^sM)\}_s.\] By naturality there is a morphism of spectral sequences $^\prime F(\infty)\to F(\infty)$, and the proof of (i) will be complete if we show it is an isomorphism on the terms on the second pages, i.e., that
\[\{\HH_i(A,\THH_j(\bb Z,M/I^sM))\}_s\isoto\{\HH_i(A/I^s,\THH_j(\bb Z,M/I^sM))\}_s\tag{\dag}\] for all $i,j\ge0$. However, B\"okstedt's calculation of $\THH(\bb Z,-)$ as torsion or cortorsion, and Corollary \ref{corollary_pro_torsion}, imply that \[\{\THH_j(\bb Z,M/I^sM)\}_s\cong \{\THH_j(\bb Z,M)\otimes_AA/I^s\}_s.\] So the map (\dag) is an isomorphism if and only if the map \[\{\HH_i(A,\THH_j(\bb Z,M)\otimes_AA/I^s)\}_s\To\{\HH_i(A/I^s,\THH_j(\bb Z,M)\otimes_AA/I^s)\}_s\] is an isomorphism; since $\THH_j(\bb Z,M)$ is a finitely generated $A$-module by B\"okstedt's caclcuation, this is indeed an isomorphism by the already established $\HH$ case of (i).

(ii): Now assume further that $A$ is $m$-AQ-finite and that $M$ is killed by $m$. Universal coefficient spectral sequences for $\HH$ assemble to a spectral sequence of pro $A$-modules \[^\prime E^2_{ij}(\infty)=\{\Tor^A_i(A/I^s,\HH_j(A,M))\}_s\implies\{\HH_{i+j}(A,M/I^sM)\}_s.\] But the the $A$-modules $\HH_j(A,M)$ are finitely generated for all $q\ge0$ by assumption and Lemma \ref{lemma_finite_generation_reduction}, so $\{\Tor^A_i(A/I^s,\HH_j(A,M))\}_s=0$ for $i\ge1$ by the Artin--Rees Theorem \ref{theorem_Artin_Rees}(i). Thus we again obtain edge map isomorphisms $\{\HH_n(A,M)\otimes_AA/I^r\}_r\isoto\{\HH_n(A,M/I^rM)\}_r$, completing the proof of isomorphism (ii) for $\HH$. The proof for $\THH$ is verbatim equivalent.
\end{proof}

We now reach our main degree-wise continuity results, which establishes Theorem \ref{theorem_intro_degreewise_continuity} of the Introduction:

\begin{theorem}\label{theorem_continuity}
Let $A$ be a Noetherian, F-finite $\bb Z_{(p)}$-algebra, $I\subseteq A$ an ideal, and $v\ge1$. Then the canonical maps
\begin{enumerate}
\item $\{\HH_n(A;\bb Z/p^v)\otimes_AA/I^s\}_s\To\{\HH_n(A/I^s;\bb Z/p^v)\}_s$
\item $\{\TR_n^r(A;\bb Z/p^v)\otimes_{W_r(A)}W_r(A/I^s)\}_s\To\{\TR_n^r(A/I^s;\bb Z/p^v)\}_s$
\end{enumerate}
are isomorphisms for all $n\ge0$, $r\ge1$.
\end{theorem}
\begin{proof}
(i): This follows from the $\HH$ part of Lemma \ref{lemma_continuity} in a relatively straightforward way: to keep the proof clear we will use $\infty$ notation for all the pro $A$-modules. As in the proof of Theorem \ref{theorem_finite_gen_of_THH_TR}(i), there is a long exact sequence of $A$-modules \[\cdots\To\HH_{n-1}(A,A[p^v])\To\HH_n(A;\bb Z/p^v)\To\HH_n(A,A/p^vA)\To\cdots,\] all of which are finitely generated. Hence we may base change by $A/I^\infty$, as in the Artin--Rees Theorem \ref{theorem_Artin_Rees}(ii), to obtain a long exact sequence of pro $A$-modules \[\hspace{-5mm}\cdots\to\HH_{n-1}(A,A[p^v])\otimes_AA/I^\infty\to\HH_n(A;\bb Z/p^v)\otimes_AA/I^\infty\to\HH_n(A,A/p^vA)\otimes_AA/I^\infty\to\cdots.\tag{1}\]

Replacing $A$ by $A/I^\infty$, and using the isomorphisms of Corollary \ref{corollary_pro_torsion} with $M=A$, there is also a long exact sequence of pro $A$-modules  \[\hspace{-1cm}\cdots\to\HH_{n-1}(A/I^\infty,A[p^v]\otimes_AA/I^\infty)\to\HH_n(A/I^\infty;\bb Z/p^v)\to\HH_n(A/I^\infty,A/p^vA\otimes_AA/I^\infty)\to\cdots.\tag{2}\] The obvious map from (1) to (2) induces an isomorphism on all the terms with coefficients in $A[p^v]$ and $A/p^vA$, by Theorem \ref{theorem_F_finite_implies_AQ_finite} and Lemma \ref{lemma_continuity}, hence also induces an isomorphism on the $\bb Z/p^v\bb Z$ terms, as desired.

(ii): If $r=1$ then the claim is that the canonical map $\{\THH_n(A;\bb Z/p^v)\otimes_AA/I^s\}_s\to\{\THH_n(A/I^s;\bb Z/p^v)\}_s$ is an isomorphism; the proof of this is verbatim equivalent to part (i). So now assume $r>1$ and proceed by induction.

We may compare the fundamental long exact sequences with finite coefficients for both $A$ and $A/I^s$ as follows:
\[\xymatrix@C=6mm{
\cdots\ar[r]& \pi_n(\THH(A;\bb Z/p^v)_{hC_{p^r}})\ar[r]\ar[d]& \TR^{r+1}_n(A;\bb Z/p^v)\ar[r]\ar[d]& \TR^r_n(A;\bb Z/p^v)\ar[r]\ar[d]& \cdots\\
\cdots\ar[r]& \pi_n(\THH(A/I^s;\bb Z/p^v)_{hC_{p^r}})\ar[r]& \TR^{r+1}_n(A/I^s;\bb Z/p^v)\ar[r]& \TR^r_n(A/I^s;\bb Z/p^v)\ar[r]& \cdots
}\]
As we saw in the proof of Theorem \ref{theorem_finite_gen_of_THH_TR}(ii), the top row consists of finitely generated $W_{r+1}(A)$-modules; so by Theorem \ref{theorem_base_change_by_Witt_rings}(i), it remains exact after base changing by the pro $W_{r+1}(A)$-algebra $W_{r+1}(A/I^\infty)$. Simultaneously assembling the bottom row into pro $W_{r+1}(A)$-modules therefore yields a map of long exact sequences of pro $W_{r+1}(A)$-modules (note that we have rotated the diagram for the sake of space):
\[\xymatrix{
\vdots & \vdots \\
\{\TR^r_n(A;\bb Z/p^v)\otimes_{W_{r+1}(A)}W_{r+1}(A/I^s)\}_s\ar[u]\ar[r]&\{\TR^r_n(A/I^s;\bb Z/p^v)\}_s\ar[u]\\
 \{\TR^{r+1}_n(A;\bb Z/p^v)\otimes_{W_{r+1}(A)}W_{r+1}(A/I^s)\}_s\ar[u]\ar[r]&\{\TR^{r+1}_n(A/I^s;\bb Z/p^v)\}_s\ar[u]\\
\{\pi_n(\THH(A;\bb Z/p^v)_{hC_{p^r}})\otimes_{W_{r+1}(A)}W_{r+1}(A/I^s)\}_s\ar[u]\ar[r]& \{\pi_n(\THH(A/I^s;\bb Z/p^v)_{hC_{p^r}})\}_s\ar[u]\\
\vdots\ar[u]&\vdots\ar[u]
}\]

\comment{
\afterpage{
\thispagestyle{plain}
\begin{sideways}
\begin{minipage}[b]{24cm}
\ul{Diagram 1}:
\[\xymatrix@C=5mm{
\cdots\ar[r]& \{\pi_n(\THH(A;\bb Z/p^v)_{hC_{p^r}})\otimes_{W_{r+1}(A)}W_{r+1}(A/I^s)\}_s\ar[r]\ar[d]& \{\TR^{r+1}_n(A;\bb Z/p^v)\otimes_{W_{r+1}(A)}W_{r+1}(A/I^s)\}_s\ar[r]\ar[d]&\{\TR^r_n(A;\bb Z/p^v)\otimes_{W_{r+1}(A)}W_{r+1}(A/I^s)\}_s\ar[r]\ar[d]& \cdots\\
\cdots\ar[r]& \{\pi_n(\THH(A/I^s;\bb Z/p^v)_{hC_{p^r}})\}_s\ar[r]& \{\TR^{r+1}_n(A/I^s;\bb Z/p^v)\}_s\ar[r]& \{\TR^r_n(A/I^s;\bb Z/p^v)\}_s\ar[r]& \cdots
}\]
\end{minipage}
\end{sideways}
\newpage
}
}
By the five lemma and inductive hypothesis, it is enough to prove that the bottom horizontal arrow in the diagram, namely \[\pi_n(\THH(A;\bb Z/p^v)_{hC_{p^r}})\otimes_{W_{r+1}(A)}W_{r+1}(A/I^\infty)\To \pi_n(\THH(A/I^\infty;\bb Z/p^v)_{hC_{p^r}}),\] is an isomorphism for all $n\ge0$. Both sides are the abutment of natural group homology spectral sequences, so it is enough to check that the map of spectral sequence is an isomorphism on the second page, namely that the canonical map \[H_i(C_{p^r},\THH_j(A;\bb Z/p^v))\otimes_{W_{r+1}(A)}W_{r+1}(A/I^\infty)\To H_i(C_{p^r},\THH_j(A/I^\infty;\bb Z/p^v))\tag{\dag}\] is an isomorphism for all $i,j\ge0$. Since $H_i(C_{p^r},\THH_j(A;\bb Z/p^v))$ is a finitely generated $A$-module, the left side of (\dag) is precisely $H_i(C_{p^r},\THH_j(A;\bb Z/p^v))\otimes_AA/I^\infty$ by Theorem \ref{theorem_base_change_by_Witt_rings}(ii); meanwhile, the right side is $H_i(C_{p^r},\THH_j(A;\bb Z/p^v)\otimes_AA/I^\infty)$ by part (ii).

Therefore it is now enough to prove that the map \[H_i(C_{p^r},\THH_j(A;\bb Z/p^v))\otimes_AA/I^\infty\To H_i(C_{p^r},\THH_j(A;\bb Z/p^v)\otimes_AA/I^\infty)\] is an isomorphism; this follows from the finite generation of $\THH_j(A;\bb Z/p^v)$ and Corollary \ref{corollary_Artin_Rees_group_homology}.
\end{proof}

From the previous degree-wise continuity theorem we obtain a description of the relationship between $\HH$, $\THH$, and $\TR^r$ of $A$ and its $I$-adic completion $\hat A=\projlim_sA/I^s$, proving Corollary \ref{corollary_intro_completing} of the Introduction: 

\begin{corollary}\label{corollary_pre_continuity}
Let $A$ be a Noetherian, F-finite $\bb Z_{(p)}$-algebra, and $I\subseteq A$ an ideal; let $\hat A$ denote the $I$-adic completion of $A$. Then all of the following maps (not just the compositions) are isomorphisms for all $n\ge0$ and $v,r\ge1$:
\begin{align*}
&\HH_n(A;\bb Z/p^v)\otimes_A\hat A\To \HH_n(\hat A;\bb Z/p^v)\To \projlim_s\HH_n(A/I^s;\bb Z/p^v)\\
&\TR^r_n(A;\bb Z/p^v)\otimes_{W_r(A)}W_r(\hat A)\To \TR^r_n(\hat A;\bb Z/p^v)\To\projlim_s\TR^r_n(A/I^s;\bb Z/p^v)
\end{align*}
\end{corollary}
\begin{proof}
We claim that each of the following canonical maps is an isomorphism:
\[\HH_n(A;\bb Z/p^v)\otimes_A\hat A\To\projlim_s\HH_n(A;\bb Z/p^v)\otimes_AA/I^s\To \projlim_s\HH_n(A/I^s;\bb Z/p^v).\] Firstly, $\HH_n(A;\bb Z/p^v)$ is a finitely generated $A$-module by Theorem \ref{theorem_finite_gen_of_THH_TR}, and $A$ is Noetherian, so standard commutative algebra, e.g.~\cite[Thm.~8.7]{Matsumura1989}, implies that the first map is an isomorphism. Secondly, Theorem \ref{theorem_continuity}(i) implies that the second map is an isomorphism.

However, Lemma \ref{lemma_conditions_for_F-finite} implies that $\hat A$ is also a Noetherian, F-finite $\bb Z_{(p)}$-algebra, so applying the same argument to $\hat A$ with respect to the ideal $I\hat A$ we obtain another composition of isomorphisms \[\HH_n(\hat A; \bb Z/p^v\bb Z)\otimes_{\hat A}\hat A\To\HH_n(\hat A;\bb Z/p^v)\otimes_{\hat A}\hat A/I^s\hat A\To \projlim_s\HH_n(\hat A/I^s\hat A;\bb Z/p^v)\] Since $\hat A/I^s\hat A\cong A/I^s$ and $\HH_n(\hat A)\otimes_{\hat A}\hat A\cong \HH_n(\hat A)$, the desired isomorphisms for $\HH$ follow.

The proofs of the isomorphisms for $\TR^r$ are exactly the same as for $\HH$, except that for $\TR^r$ one must also note that $W_r(A)$ is Noetherian by Langer--Zink (Theorem \ref{theorem_Langer_Zink}) and use Lemma \ref{lemma_W_S_of_completion}.
\end{proof}

Whereas the previous two continuity results have concerned individual groups, we now prove the spectral continuity of $\THH$, $\TR^r$, etc.~under our usual hypotheses; this establishes Theorem \ref{theorem_intro_continuity_in_complete_case} of the Introduction:

\begin{theorem}\label{theorem_continuity_in_complete_case}
Let $A$ be a Noetherian, F-finite $\bb Z_{(p)}$-algebra, and $I\subseteq A$ an ideal; assume that $A$ is $I$-adically complete. Then, for all $1\le r\le\infty$, the following canonical maps of spectra are weak equivalences after $p$-completion:
\begin{align*}
&\TR^r(A;p)\To\holim_s\TR^r(A/I^s;p)\\
&\TC^r(A;p)\To\holim_s\TC^r(A/I^s;p)
\end{align*}
\end{theorem}
\begin{proof}
We must prove that each of the given maps of spectra is a weak equivalence with $\bb Z/p^v\bb Z$-coefficients for all $v\ge1$. Firstly, fixing $1\le r<\infty$, the homotopy groups of $\holim_s\TR^r(A/I^s;\bb Z/p^v)$ fit into short exact sequences \[0\to{\projlim_s}^1\TR^r_{n+1}(A/I^s;\bb Z/p^v)\to\pi_n(\holim_s\TR^r(A/I^s;\bb Z/p^v))\to \projlim_s\TR^r_n(A/I^s;\bb Z/p^v)\to 0.\] Theorem \ref{theorem_continuity}(i) implies that the left-most term is $\projlim_s^1\TR^r_{n+1}(A;\bb Z/p^v)\otimes_{W_r(A)}W_r(A/I^s)$, which vanishes because of the surjectivity of the transition maps in the pro abelian group $\{\TR^r_{n+1}(A;\bb Z/p^v)\otimes_{W_r(A)}W_r(A/I^s)\}_s$. In conclusion, the natural map \[\pi_n(\holim_s\TR^r(A/I^s;\bb Z/p^v))\To \projlim_s\TR^r_n(A/I^s;\bb Z/p^v)\] is an isomorphism for all $n\ge0$. But since $A$ is already $I$-adically complete, Corollary \ref{corollary_pre_continuity} states that $\TR^r_n(A;\bb Z/p^v)\to\projlim_s\TR^r_n(A/I^s;\bb Z/p^v)$ is also an isomorphism for all $n\ge0$. So the map $\TR^r(A;\bb Z/p^v)\to\holim_s\TR^r(A/I^s;\bb Z/p^v)$ induces an isomorphism on all homotopy groups, as required to prove the first weak equivalence.

The claims for $\TC^r$, $\TR^\infty=\TR$, and $\TC^\infty=\TC$ then follow since homotopy limits commute.
\end{proof}

In the remainder of this section we consider straightforward consequences of the previous three results in special situations. We begin with the case in which $p$ is nilpotent in $A$:

\begin{corollary}\label{corollary_nilpotent_continuity}
Let $A$ be a Noetherian, F-finite $\bb Z_{(p)}$-algebra, and $I\subseteq A$ an ideal; assume that $p$ is nilpotent in $A$, and let $\hat A$ denote the $I$-adic completion of $A$. Then all of the following maps (not just the compositions) are isomorphisms for all $n\ge0$, $r\ge1$:
\begin{itemize}
\item[]$\{\HH_n(A)\otimes_AA/I^s\}_s\To\{\HH_n(A/I^s)\}_s$
\item[]$\{\TR_n^r(A;p)\otimes_{W_r(A)}W_r(A/I^s)\}_s\To\{\TR_n^r(A/I^s;p)\}_s$
\item[]$\HH_n(A)\otimes_A\hat A\To \HH_n(\hat A)\To \projlim_s\HH_n(A/I^s)$
\item[]$\TR^r_n(A;p)\otimes_{W_r(A)}W_r(\hat A)\To \TR^r_n(\hat A;p)\To\projlim_s\TR^r_n(A/I^s;p)$
\end{itemize}
Moreover, the maps of spectra in the statement of Theorem \ref{theorem_continuity_in_complete_case} are weak equivalences without $p$-completing.

\end{corollary}
\begin{proof}
Note that $p$ is also nilpotent in $W_r(A)$, by Lemma \ref{lemma_F-finiteness_of_Witt_ring}(ii). So, fixing $r\ge1$, we may pick $v\gg0$ such that the groups \[\HH_n(A),\;\HH_n(A/I^s),\;\THH_n(A),\;\THH(A/I^s),\;\TR_n^r(A;p),\;\TR^r_n(A/I^s;p)\] are annihilated by $p^v$ for all $n\ge0$, $s\ge1$. Hence the spectra appearing in Theorem \ref{theorem_continuity_in_complete_case} are all $p$-complete, and the isomorphisms follow from Theorem \ref{theorem_continuity} and Corollary \ref{corollary_pre_continuity} by examining the short exact sequences for homotopy groups with finite coefficients, in the usual way.
\end{proof}

\begin{remark}
Some of the statements of Corollary \ref{corollary_nilpotent_continuity} hold for rings other than Noetherian, F-finite $\bb Z_{(p)}$-algebras in which $p$ is nilpotent. In particular, if $A$ is a Noetherian ring which is AQ-finite over $\bb Z$ (e.g., an essentially finite type $\bb Z$-algebra) and $I\subseteq A$ is an ideal, then the pro $\HH$ and $\THH=\TR^1$ isomorphisms of Corollary \ref{corollary_nilpotent_continuity} hold; indeed, this follows immediately from Lemma \ref{lemma_continuity} with $m=0$ and $M=A$.

Suppose now, in addition to be being Noetherian and AQ-finite, that $A$ is an F-finite $\bb Z_{(p)}$-algebra (e.g., an essentially finite type $\bb Z_{(p)}$-algebra). Then the pro $\TR^r$ isomorphisms of Corollary \ref{corollary_nilpotent_continuity} hold: this is proved by verbatim repeating the proof of Theorem \ref{theorem_continuity}(ii) integrally instead of with finite coefficients.
\end{remark}

Now, in stark contrast with the case in which $p$ is nilpotent, we consider the case where $I=pA$; here our methods yield a new proof, albeit under different hypotheses, of a result of Geisser and Hesselholt, as we will discuss in Remark \ref{remark_on_GH_p}:

\begin{corollary}\label{corollary_p_I}
Let $A$ be a Noetherian, F-finite $\bb Z_{(p)}$-algebra. Then the following canonical maps are isomorphisms for all $n\ge0$ and $v,r\ge1$:
\begin{itemize}
\item[] $\HH_n(A;\bb Z/p^v)\To\{\HH_n(A/p^sA;\bb Z/p^v)\}_s$
\item[] $\TR_n^r(A;\bb Z/p^v)\To\{\TR_n^r(A/p^sA;\bb Z/p^v)\}_s$
\item[] $\TC_n^r(A;\bb Z/p^v)\To\{\TC_n^r(A/p^sA;\bb Z/p^v)\}_s$
\end{itemize}
Moreover, for $1\le r\le\infty$, all of the following maps (not just the compositions) of spectra are weak equivalences after $p$-completion:
\begin{itemize}
\item[] $\TR^r(A;p)\To\TR^r(A_p^\comp;p)\To\holim_s\TR^r(A/p^sA;p)$,
\item[]$\TC^r(A;p)\To\TC^r(A_p^\comp;p)\To\holim_s\TC^r(A/p^sA;p)$.
\end{itemize}
\end{corollary}
\begin{proof}
The groups $\HH_n(A;\bb Z/p^v)$ and $\TR_n^r(A;\bb Z/p^v)$ are annihilated by $p^v$. Since we proved in Lemma \ref{lemma_W_r_of_p_completion} that $p^vW_r(A)$ contains $W_r(p^sA)$ for $s\gg0$, we deduce that
\begin{itemize}
\item[]$\{\HH_n(A;\bb Z/p^v)\otimes_AA/p^sA\}_s=\HH_n(A;\bb Z/p^v)$,
\item[]$\{\TR_n^r(A;\bb Z/p^v)\otimes_{W_r(A)}W_r(A/p^sA)\}_s=\TR_n^r(A;\bb Z/p^v)$.
\end{itemize}
Hence the desired pro $\HH$ and $\TR^r$ isomorphisms follow from Theorem \ref{theorem_continuity}. The pro $\TC$ isomorphism then follows in the usual way by applying the five lemma to the long exact sequence relating $\TC^r$, $\TR^r$, and $\TR^{r-1}$.

Since $A_p^\comp$ is also a Noetherian, F-finite $\bb Z_{(p)}$-algebra by Lemma \ref{lemma_conditions_for_F-finite}, and since $A_p^\comp/p^sA_p^\comp\cong A/p^sA$, applying isomorphism (ii) to both $A$ and $A_p^\comp$ yields \[\TR^r_n(A;\bb Z/p^v)\cong\{\TR^r_n(A/p^s;\bb Z/p^v)\}_s\cong \TR^r_n(A_p^\comp;\bb Z/p^v)\] for any integer $r\ge1$. This proves that $\TR^r(A;p)_p^\comp\simeq\TR^r(A_p^\comp;p)_p^\comp$; the same follows for $\TC^r$, $\TR$ and $\TC$ by taking homotopy limits.

The second column of maps between spectra are weak equivalences after $p$-completing by Theorem~\ref{theorem_continuity_in_complete_case}.
\end{proof}

\begin{remark}\label{remark_on_GH_p}
The isomorphisms of Corollary \ref{corollary_p_I} were proved by Geisser and Hesselholt \cite[\S3]{GeisserHesselholt2006a} for any (possibly non-commutative) ring $A$ in which $p$ is a non-zero divisor. The key assertion is the pro $\THH$ isomorphism, namely \[\THH_n(A;\bb Z/p^v)\isoto\{\THH_n(A/p^sA;\bb Z/p^v)\}_s,\] as they have a short argument to deduce the corresponding $\TR^r$ isomorphism from this by induction.

This $\THH$ isomorphism can also be proved using our methods as follows:  mimicking the proof of Lemma \ref{lemma_continuity} via Proposition \ref{proposition_Bjorns_SS} and the Artin--Rees vanishing result $\{\Tor_n^A(A/I^s,A/I^s)\}_s=0$, it is enough to show that \[\{\Tor_n^A(A/p^sA,A/p^sA)\}_s=0\] for all $n>0$ whenever $p$ is a non-zero divisor of a possibly non-commutative ring $A$. But in such a situation we may calculate Tor using the projective resolution $0\to A\xto{\times p^s}A\to A/p^sA\to 0$ of $A/p^sA$, and from this calculation it easily follows that the map \[\Tor_1^A(A/p^{2s}A,A/p^{2s}A)\To \Tor_1^A(A/p^sA,A/p^sA)\] is zero, as required.
\end{remark}

\subsection{The pro Hochschild--Kostant--Rosenberg theorems}\label{subsection_pro_HKR}
Given a geometrically regular (e.g., smooth) morphism $k\to A$ of Noetherian rings, the classical Hochschild--Kostant--Rosenberg theorem \cite[Thm.~3.4.4]{Loday1992} states that the antisymmetrisation map $\Omega_{A/k}^n\to HH_n^k(A)$ is an isomorphism of $A$-modules for all $n\ge0$.

\begin{remark}\label{remark_geometrically_regular}
Since the notion of a geometrically regular morphism may not be familiar to all readers, here we offer a brief explanation. A good reference is R.~Swan's exposition of Neron--Popescu desingularisation \cite{Swan1998}.

If $k$ is a field, then a Noetherian $k$-algebra $A$ is said to be {\em geometrically regular} over $k$, or that $k\to A$ is a geometrically regular morphism, if and only if $A\otimes_kk'$ is a regular ring for all finite field extensions $k'/k$. If $k$ is perfect then this is equivalent to $A$ being a regular ring, which is equivalent to $A$ being smooth over $k$ if we moreover assume that $A$ is essentially of finite type over $k$. If $k$ is no longer necessarily a field, then $k\to A$ is said to be geometrically regular if and only if it is flat and $k(\frak p)\to R\otimes_kk(\frak p)$ is geometrically regular in the previous sense for all prime ideals $\frak p\subseteq k$, where $k(\frak p)=k_\frak p/\frak pk_\frak p$.

The Neron--Popescu desingularisation theorem \cite{Popescu1985, Popescu1986} states that if $A$ is a $k$-algebra, with both rings Noetherian, then $A$ is geometrically regular over $k$ if and only if it is a filtered colimit of smooth, finite-type $k$-algebras.
\end{remark}

The following establishes the pro Hochschild--Kostant--Rosenberg theorem for algebraic Hochschild homology in full generality:

\begin{theorem}[Pro HKR theorem for Hochschild homology]\label{theorem_pro_HKR}
Let $k\to A$ be a geometrically regular morphism of Noetherian rings, and $I\subseteq A$ an ideal. Then the canonical map of pro $A$-modules \[\{\Omega^n_{(A/I^s)/k}\}_s\To\{\HH_n^k(A/I^s)\}_s\] is an isomorphism for all $n\ge0$.
\end{theorem}
\begin{proof}
Consider the following commutative diagram of pro $A$-modules, in which the vertical arrows are the antisymmetrisation maps:
\[\xymatrix{
\{A/I^s\otimes_A\HH_n^k(A)\}_s \ar[r]^{(1)} & \{\HH_n^k(A,A/I^s)\}_s \ar[r]^{(2)} & \{\HH_n^k(A/I^s)\}_n\\
\{A/I^s\otimes_A\Omega_{A/k}^n\}_s \ar[u]\ar[rr]^{(3)} && \{\Omega^n_{(A/I^s)/k}\}_s \ar[u]
}\]
As recalled above, the classical HKR theorem implies that the antisymmetrization map $\Omega_{A/k}^j\to\HH_j^k(A)$ is an isomorphism for all $j\ge0$. So the left vertical arrow is an isomorphism. Moreover, Neron--Popescu desingularisation implies that $A$ is a filtered colimit of smooth, finite type $k$-algebras, and so $\Omega_{A/k}^j\cong\HH_j^k(A)$ is a filtered colimit of free $A$-modules, hence is a flat $A$-module.

For any $A$-module $M$, the universal coefficient spectral sequence $\Tor_i^A(M,\HH_j^k(A))\Rightarrow\HH_{i+j}(A,M)$ therefore collapses to edge map isomorphisms $M\otimes_A\Omega_{A/k}^n\isoto\HH_n^k(A,M)$. In particular, taking $M=A/I^s$ shows that arrow (1) is an isomorphism.

Next, Lemma \ref{lemma_continuity}(i) states that arrow (2) is an isomorphism (to be precise, Lemma \ref{lemma_continuity}(i) was stated only for the ground ring $\bb Z$, but the proof worked verbatim for any ground ring $k$). Finally, arrow (3) is easily seen to be an isomorphism using the inclusion $d(I^{2s})\subseteq I^s\Omega_{A/k}^n$.

It follows that the right vertical arrow is also an isomorphism, as desired.
\end{proof}

\begin{corollary}[Pro HKR Theorem for cyclic homology]
Let $k\to A$ be a geometrically regular morphism of Noetherian rings, and $I\subseteq A$ an ideal. Then there is a natural spectral sequence of pro $k$-modules \[E_{pq}^2=\begin{cases}\{\Omega_{(A/I^s)/k}^n/d\Omega_{(A/I^s)/k}^{n-1}\}_s&p=0\\\{H_\sub{dR}^{q-p}((A/I^s)/k)\}_s&p>0\end{cases}\implies \{HC_{p+q}^k(A/I^s)\}_s.\] If $k$ contains $\bb Q$ then this degenerates with naturally split filtration, yielding \[\{HC_n^k(A/I^s)\}_s\cong\{\Omega_{(A/I^s)/k}^n/d\Omega_{(A/I^s)/k}^{n-1}\}_s\oplus\bigoplus_{0\le p<\tfrac{n}{2}}\{H_\sub{dR}^{n-2p}((A/I^s)/k)\}_s.\]
\end{corollary}
\begin{proof}
This follows by combining Theorem \ref{theorem_pro_HKR} with standard arguments in cyclic homology.
\end{proof}

\begin{remark}
In the special case of certain finite type algebras over fields, the pro HKR theorem for algebraic Hochschild homology was established by G.~Corti\~nas, C.~Haesemeyer, and C.~Weibel \cite[Thm.~3.2]{Cortinas2009}. This full version of the pro HKR theorem has recently been required in the study of infinitesimal deformations of algebraic cycles \cite{BlochEsnaultKerz2013, Morrow_Deformational_Hodge}.
\end{remark}

Next we turn to topological Hochschild homology, for which we must first briefly review the de Rham--Witt complex. Given an $\bb F_p$-algebra $A$, the existence and theory of the $p$-typical de Rham--Witt complex $W_r\Omega_A^\bullet$, which is a pro differential graded $W(A)$-algebra, is due to S.~Bloch, P.~Deligne, and L.~Illusie; see especially \cite[Def.~I.1.4]{Illusie1979}. It was later extended by Hesselholt and Madsen to $\bb Z_{(p)}$-algebras with $p$ odd, and finally by Hesselholt in full generality; see the introduction to \cite{Hesselholt2010} for further discussion. We will only require the classical formulation for $\bb F_p$-algebras, with which we assume the reader has some familiarity.

We may now recall Hesselholt's topological Hochschild homology analogue of the HKR theorem. Letting $A$ be an $\bb F_p$-algebra, one knows that the pro graded ring $\{\TR^r_\blob(A;p)\}_r$ is a $p$-typical Witt complex with respect to its operators $F,V,R$; by universality of the de Rham--Witt complex, there are therefore natural maps of graded $W_r(A)$-algebras \cite[Prop.~1.5.8]{Hesselholt1996} \[W_r\Omega_A^\blob\To TR_\blob^r(A;p)\] for $r\ge0$, which are compatible with the Frobenius,  Verschiebung, and Restriction maps (in other words, a morphism of $p$-typical Witt complexes). Since there is also a natural map of graded $W_r(\bb F_p)$-algebras $TR_\blob^r(\bb F_p;p)\to  TR_\blob^r(A;p)$, we may tensor these algebra maps to obtain a natural morphism of graded $W_r(A)$-algebras \[W_r\Omega_A^\blob\otimes_{W_r(\bb F_p)}TR_\blob^r(\bb F_p;p)\to  TR_*^r(A;p).\tag{\dag}\] Hesselholt's HKR theorem may be stated as follows:

\begin{theorem}[{Hesselholt \cite[Thm.~B]{Hesselholt1996}}]
Suppose that the $\bb F_p$-algebra $A$ is regular. Then (\dag) is an isomorphism of graded $W_r(A)$-algebras for all $r\ge1$.

Moreover, there are isomorphisms $W_r(\bb F_p)[\sigma_r]\cong TR_\blob^r(\bb F_p;p)$ of graded $W_r(\bb F_p)$-algebras, where the polynomial variable $\sigma_r$ has degree $2$, $F(\sigma_r)=\sigma_{r-1}$, $V(\sigma_r)=p\sigma_{r+1}$, and $R(\sigma_r)=p\lambda_r\sigma_{r-1}$ for some unit $\lambda_r\in W_r(\bb F_p)$.
\end{theorem}

We now prove the pro Hochschild--Kostant--Rosenberg theorem for $\THH$ and $\TR^r$; since infinite direct sums do not commute with the formation of pro-abelian groups, we must state it degree-wise:

\begin{theorem}[Pro HKR Theorem for $\THH$ and $\TR^r$]\label{theorem_top_pro_HKR}
Let $A$ be a regular, F-finite $\bb F_p$-algebra, and $I\subseteq A$ an ideal. Then the canonical map \[\bigoplus_{i=0}^n\{W_r\Omega_{A/I^s}^i\otimes_{W_r(\bb F_p)}\TR^r_{n-i}(\bb F_p;p)\}_s\To \{\TR^r_n(A/I^s;p)\}_s\] of pro $W_r(A)$-modules is an isomorphism for all $n\ge0$, $r\ge1$.
\end{theorem}
\begin{proof}
Consider the following commutative diagram of pro $W_r(A)$-modules:
\[\xymatrix@C=4mm{
\displaystyle\bigoplus_{i=0}^n\{W_r(A/I^s)\otimes_{W_r(A)}W_r\Omega_A^i\otimes_{W_r(\bb F_p)}\TR_{n-i}^r(\bb F_p;p)\}_s\ar[r]\ar[d] & \{W_r(A/I^s)\otimes_{W_r(A)}\TR_n^r(A;p)\}_s\ar[d]\\
\bigoplus_{i=0}^n\{W_r\Omega_{A/I^s}^i\otimes_{W_r(\bb F_p)}\TR^r_{n-i}(\bb F_p;p)\}_s \ar[r]& \{\TR^r_n(A/I^s;p)\}_s
}\]
Hesselholt's HKR theorem implies that the top horizontal arrow is an isomorphism. Corollary \ref{corollary_nilpotent_continuity} implies that the right vertical arrow is an isomorphism. Since we wish to establish that the bottom horizontal arrow is an isomorphism, it is now enough to show that the left vertical arrow is an isomorphism, for which it suffices to prove that the canonical map \[\{W_r(A/I^s)\otimes_{W_r(A)}W_r\Omega_A^i\}_s\To\{W_r\Omega_{A/I^s}^i\}_s\tag{\dag}\] of pro $W_s(A)$-modules is an isomorphism for all $i\ge0$. To show this, note that the maps in (\dag) are surjective, so one needs only to show that the pro abelian group arising from the kernels is zero. This is an easy consequence of Lemma \ref{lemma_witt_1}(iv) and the Leibnitz rule; see, e.g.,~\cite[Prop.~2.5]{GeisserHesselholt2006b}.
\end{proof}

Next we let $r\to\infty$ to prove the pro Hochschild--Kostant--Rosenberg theorem for $\TR$, or more precisely for the pro spectrum $\{\TR^r\}_r$. This takes the form of an isomorphism of {\em pro pro abelian groups}; that is, an isomorphism in the category $\op{Pro}(\op{Pro}Ab)$. Although this initially appears unusual, we believe it is necessary to state the result in its strongest form; however, we note that, by taking the diagonal in the result, we also obtain a weaker isomorphism of pro abelian groups $\{W_r\Omega_{A/I^r}^n\}_r\isoto \big\{\TR^r_n(A/I^r;p)\}_r$.

\begin{corollary}[Pro HKR Theorem for $\TR$]\label{corollary_top_pro_HKR_for_TR}
With notation as in the previous theorem, the canonical map of pro pro abelian groups \[\big\{\{W_r\Omega_{A/I^s}^n\}_s\big\}_r\To \big\{\{\TR^r_n(A/I^s;p)\}_s\big\}_r\] is an isomorphism for all $n\ge0$.
\end{corollary}
\begin{proof}
We stated as part of Hesselholt's HKR theorem that the Restriction map acts on $\TR_*^r(\bb F_p;p)\cong W_r(\bb F_p)[\sigma_r]$ by sending $\sigma_r$ to $p\lambda_r\sigma_{r-1}$ for some $\lambda_r\in W_r(\bb F_p)^\times$. In particular, it follows that $R^r:\TR_n^{2r}(\bb F_p;p)\to\TR_n^r(\bb F_p;p)$ is zero for all $n,r\ge1$. Thus the inverse system of pro abelian groups
\[\cdots \xto{R}\{W_{r+1}\Omega_{A/I^s}^i\otimes_{W_{r+1}(\bb F_p)}\TR^{r+1}_{n-i}(\bb F_p;p)\}_s\xto{R} \{W_r\Omega_{A/I^s}^i\otimes_{W_r(\bb F_p)}\TR^r_{n-i}(\bb F_p;p)\}_s\xto{R}\cdots\] is trivial Mittag-Leffler unless $i=n$. So the isomorphisms of Theorem \ref{theorem_top_pro_HKR} assemble over $r\ge1$ to the desired isomorphism.
\end{proof}

\section{Proper schemes over an affine base}\label{section_proper}
In this section we extend the finite generation results of Section \ref{section_finite_generation} and the continuity results of Section \ref{section_continuity_and_HKR} to proper schemes over an affine base. The key idea is to combine the already-established results with Zariski descent and Grothendieck's formal function theorem for coherent cohomology, which we will recall in Theorem \ref{theorem_GFFT} for convenience.

We first review the scheme-theoretic versions of $\THH$, $\TR^r$, etc. For a quasi-compact, quasi-separated scheme $X$, the spectra $\THH(X)$, $\TR^r(X;p)$, $\TC^r(X;p)$, $\TR(X;p)$, and $\TC(X;p)$ were defined by Geisser and Hesselholt in \cite{GeisserHesselholt1999} in such a way that all these presheaf of spectra satisfy Zariski descent (see the proof of \cite[Corol.~3.3.3]{GeisserHesselholt1999}). In particular, assuming that $X$ has finite Krull dimension, there is a bounded spectral sequence \[E_2^{ij}=H^i(X,\cal\THH_{-j}(X))\implies \THH_{-i-j}(X)\] where the caligraphic notation $\cal\THH_n(X)$ denotes the sheafification of the Zariski presheaf on $X$ given by $U\mapsto\THH_n(\roi_X(U))$; the same applies to $\TR^r(-;p)$ and $\TC^r(-;p)$, and we will always use caligraphic notation to denote such Zariski sheafifications, including when working with finite coefficients.

Moreover, the relations between the theories in the affine case explained in Section~\ref{subsection_intro_to_THH} continue to hold for schemes \cite[Prop.~3.3.2]{GeisserHesselholt1999}, and analogous comments also apply to algebraic Hochschild homology, thanks to Weibel \cite{Weibel1991, Weibel1996}.

We secondly recall the scheme-theoretic version of Witt vectors; further details may be found in the appendix of \cite{LangerZink2004}. Given a ring $A$ and an element $f\in A$, there is a natural isomorphism $W_r(A_f)\cong W_r(A)_{[f]}$ where $[f]$ is the Teichm\"uller lift of $f$; this localisation result means that, for any scheme $X$, we may define a new scheme $W_r(X)$ by applying $W_r$ locally. The restriction map $R^{r-1}:W_r(X)\to X$ induces a closed embedding of schemes $R^{r-1}:X\into W_r(X)$, which is an isomorphism of the underlying topological spaces if $p$ is nilpotent on $X$. Particularly in the presence of F-finiteness (a $\bb Z_{(p)}$-scheme is said to be F-finite if and only if it has a finite cover by spectra of F-finite $\bb Z_{(p)}$-algebras), many properties of $X$ are inherited by $W_r(X)$; see Prop.~A.1 -- Corol.~A.7 of \cite{LangerZink2004}:
\begin{enumerate}\itemsep0pt
\item $X$ separated $\implies$ $W_r(X)$ is separated.
\item $X$ Noetherian and F-finite $\implies$ $W_r(X)$ is Noetherian.
\item $X\to Y$ of finite type and $Y$ F-finite $\implies$ $X$ is F-finite and $W_r(X)\to W_r(Y)$ is of finite type.
\item $X\to Y$ proper and $Y$ F-finite $\implies$ $W_r(X)\to W_r(Y)$ is proper.
\end{enumerate}

Our first aim is to prove that the sheaves arising from $\HH$, $\THH$, and $\TR^r$ are quasi-coherent; this is a standard result for $\HH$ and $\THH$, but we could not find a reference covering the $\TR^r$ sheaves. We note that the following lemma remains true when working with finite $\bb Z/p^v\bb Z$-coefficients; this follows from the lemma as stated using the usual short exact sequences for finite coefficients.

\begin{lemma}\label{lemma_quasi_coherent}
Let $X$ be a quasi-compact, quasi-separated scheme. Then:
\begin{enumerate}
\item $\cal \HH_n(X)$ and $\cal \THH_n(X)$ are quasi-coherent sheaves on $X$ for all $n\ge0$.
\item If $X$ is moreover an F-finite, $\bb Z_{(p)}$-scheme, then $R^{r-1}_*\cal \TR_n^r(X;p)$ is a quasi-coherent sheaf on $W_r(X)$ for all $n\ge0$, $r\ge1$.
\end{enumerate}
\end{lemma}
\begin{proof}
Let $\Spec R\subseteq X$ be any affine open subscheme of $X$; we must show that for any $f\in R$, the canonical maps
\begin{align*}
\HH_n(R)\otimes_RR_f&\to\HH_n(R_f),\\
\TR_n^r(R;p)\otimes_{W_r(R)}W_r(R_f)&\to\TR_n^r(R_f;p),
\end{align*}
are isomorphisms for all $n\ge0$, $r\ge1$ (assuming in addition that $R$ is an F-finite $\bb Z_{(p)}$-algebra in the case of $\TR_n^r$, $r>1$).

Firstly, $R_f$ is flat over $R$, so the spectral sequence of Proposition \ref{proposition_Bjorns_SS} degenerates to edge map isomorphisms \[\HH_n(R,R_f)\isoto \HH_n(R_f,R_f\otimes_RR_f).\] But $R_f\otimes_RR_f=R_f$, and its flatness over $R$ implies that $\HH_n(R,R_f)=\HH_n(R)\otimes_RR_f$; this proves the claim for $\HH$.

The claim for $\THH=\TR^1$ can be proved either using the $\THH$ version of the spectral sequence of Proposition \ref{proposition_Bjorns_SS}, or in a straightforward way using the Pirashvili--Waldhausen spectral sequence and the already-proved $\HH$ claim; we omit the details.

Finally, assuming that $R$ is an F-finite $\bb Z_{(p)}$-algebra, we must prove the claim for $\TR^r_n$ by induction on $r$; we have just established the case $r=1$. Since $W_r(R_f)=W_r(R)_{[f]}$ is flat over $W_r(R)$, we may base change by $W_{r+1}(R_f)$ the fundamental long exact sequence of Section \ref{subsection_TR}(i) for $W_{r+1}(R)$, and compare it to the long exact sequence for $R_f$; this yields the following diagram with exact columns:
\[\xymatrix{
\vdots&\vdots\\
\TR^r_n(R;p)\otimes_{W_r(R)}W_r(R_f)\ar[u]\ar[r]&\TR^r_n(R_f;p)\ar[u]\\
\TR^{r+1}_n(R;p)\otimes_{W_{r+1}(R)}W_{r+1}(R_f)\ar[u]\ar[r]&\TR^{r+1}_n(R_f;p)\ar[u]\\
\pi_n(\THH(R)_{hC_{p^r}})\otimes_{W_{r+1}(R)}W_{r+1}(R_f)\ar[u]\ar[r]&\pi_n(\THH(R_f)_{hC_{p^r}})\ar[u]\\
\vdots\ar[u]&\vdots\ar[u]
}\]

By the five lemma and induction on $r$, it is sufficient to prove that the bottom horizontal arrow is an isomorphism. Moreover, the domain and codomain of this map are compatibly described by group homology spectral sequences, namely \[H_i(C_{p^r},\THH_j(R))\otimes_{W_{r+1}(R)}W_{r+1}(R_f)\implies \pi_{i+j}(\THH(R)_{hC_{p^r}})\otimes_{W_{r+1}(R)}W_{r+1}(R_f)\] and \[H_i(C_{p^r},\THH_j(R_f))\implies \pi_{i+j}(\THH(R_f)_{hC_{p^r}}),\] so it is enough to prove that the canonical map \[H_i(C_{p^r},\THH_j(R))\otimes_{W_{r+1}(R)}W_{r+1}(R_f)\To H_i(C_{p^r},\THH_j(R_f))\] is an isomorphism for all $i,j\ge0$ and $r\ge1$. Since $W_{r+1}(R_f)$ is flat over $W_{r+1}(R)$, we may identify the left side with $H_i(C_{p^r},\THH_j(R)\otimes_{W_{r+1}(R)}W_{r+1}(R_f))$, and so it remains only to show that the map $\THH_j(R)\otimes_{W_{r+1}(R)}W_{r+1}(R_f)\to\THH_j(R_f)$ is an isomorphism, where $W_{r+1}(R)$ is acting on $\THH_j(R)$ via $F^r:W_{r+1}(R)\to R$; but this follows from the fact that the diagram
\[\xymatrix{
W_{r+1}(R)\ar[r]\ar[d]_{F^r} & W_{r+1}(R_f)\ar[d]^{F^r}\\
R\ar[r] & R_f
}\]
is cocartesian by \cite[Lem.~A.18]{LangerZink2004} (this is where we require that $R$ be an F-finite $\bb Z_{(p)}$-algebra).
\end{proof}

As a consequence of the lemma and our earlier finite generation results, we therefore obtain the fundamental coherence results we will need:

\begin{corollary}\label{corollary_coherence_of_THH_TR}
Let $A$ be a Noetherian, F-finite $\bb Z_{(p)}$-algebra, $X$ an essentially finite-type $A$-scheme, and $v\ge1$. Then:
\begin{enumerate}
\item $\cal \HH_n(X;\bb Z/p^v)$ and $\cal \THH_n(X;\bb Z/p^v)$ are coherent sheaves on $X$ for all $n\ge0$.
\item $R^{r-1}_*\cal \TR^r_n(X;\bb Z/p^v)$ is a coherent sheaf on $W_r(X)$ for all $n\ge0$, $r\ge1$
\end{enumerate}
\end{corollary}
\begin{proof}
The specified sheaves are quasi-coherent by the finite coefficient version of Lemma \ref{lemma_quasi_coherent}, so we must prove that if $\Spec R$ is an affine open subscheme of $X$, then $\HH_n(R;\bb Z/p^v)$ and $\THH_n(R;\bb Z/p^v)$ are finitely generated $R$-modules, and that $\TR_n^r(R;\bb Z/p^v)$ is a finitely generated $W_r(R)$-module. But Lemma \ref{lemma_conditions_for_F-finite} implies that $R$ is also a Noetherian, F-finite, $\bb Z_{(p)}$-algebra, so this follows from Theorem \ref{theorem_finite_gen_of_THH_TR}.
\end{proof}

We now obtain the generalisation to proper schemes of our earlier finite generation results:

\begin{theorem}\label{theorem_finite_gen_proper}
Let $A$ be a Noetherian, F-finite, finite Krull-dimensional $\bb Z_{(p)}$-algebra, $X$ a proper scheme over $A$, and $v\ge1$. Then:
\begin{enumerate}
\item $\HH_n(X;\bb Z/p^v)$ and $\THH_n(X;\bb Z/p^v)$ are finitely generated $A$-modules for all $n\ge~0$.
\item $\TR_n^r(X;\bb Z/p^v)$ is a finitely generated $W_r(A)$-module for all $n\ge0$, $r\ge1$.
\end{enumerate}
\end{theorem}
\begin{proof}
The coherence assertion of Corollary \ref{corollary_coherence_of_THH_TR} and the properness of $X$ over $A$ implies that $H^i(X,\cal\HH_n(X;\bb Z/p^v))$ and $H^i(X,\cal \THH_n(X;\bb Z/p^v))$ are finitely generated $A$-modules, and $H^i(X,\cal\TR^r_n(X;\bb Z/p^v))$ is a finitely generated $W_r(X)$-module, for all $i,n\ge0$ and $r\ge1$.

Since the presheafs of spectra $\HH\wedge S/p^v$, $\THH\wedge S/p^v$, and $\TR^r(-;p)\wedge S/p^v$ satisfy Zariski descent by the results of Geisser--Hesselholt recalled at the start of the section, and since $X$ has finite Krull dimension by assumption, there are right half-plane, bounded, Zariski descent spectral sequences (with differentials $E^{ij}_r\to E^{i+r\,j-r+1}_r$)
\[E_2^{ij}=H^i(X,\cal\HH_{-j}(X;\bb Z/p^v))\implies \HH_{-i-j}(X;\bb Z/p^v),\]
\[E_2^{ij}=H^i(X,\cal\THH_{-j}(X;\bb Z/p^v))\implies \THH_{-i-j}(X;\bb Z/p^v),\]
\[E_2^{ij}=H^i(X,\cal\TR_{-j}^r(X;\bb Z/p^v))\implies \TR_{-i-j}^r(X;\bb Z/p^v),\]
of modules over $A$, $A$, and $W_r(A)$ respectively. This evidently completes the proof.
\end{proof}

By $p$-completing we obtain Theorem \ref{theorem_intro_fin_gen_p_complete} of the Introduction:

\begin{corollary}\label{corollary_p_complete_finite_gen_proper}
Let $A$ be a Noetherian, F-finite, finite Krull-dimensional $\bb Z_{(p)}$-algebra, and $X$ a proper scheme over $A$. Then:
\begin{enumerate}
\item $\HH_n(X;\bb Z_p)$ and $\THH_n(X;\bb Z_p)$ are finitely generated $A_p^\comp$-modules for all $n\ge~0$.
\item $\TR_n^r(X;p,\bb Z_p)$ is a finitely generated $W_r(A_p^\comp)$-module for all $n\ge0$, $r\ge1$.
\end{enumerate}
\end{corollary}
\begin{proof}
This follows from Theorem \ref{theorem_finite_gen_proper} via Proposition \ref{proposition_to_prove_finite_gen_p_adically} in the appendix.
\end{proof}

Next we generalise our degree-wise continuity result of Theorem \ref{theorem_continuity} to the case of a proper scheme. To do this we will consider a proper scheme $X$ over an affine base $A$, fix an ideal $I\subseteq X$, and always write $X_s:=X\times_AA/I^s$. The following theorem of Grothendieck will be required:

\begin{theorem}[{Grothendieck's Formal Functions Theorem \cite[Cor.~4.1.7]{EGA_III_I}}]\label{theorem_GFFT}
Let $A$ be a Noetherian ring, $I\subseteq A$ an ideal, $X$ a proper scheme over $A$, and $N$ a coherent $\roi_X$-module. Then the canonical map of pro $A$-modules \[\{H^n(X,N)\otimes_AA/I^s\}_s\isoto\{H^n(X_s,N/I^sN)\}_s\] is an isomorphism for all $n\ge0$.
\end{theorem}
\begin{proof}[Proof/Remark]
In fact, Grothendieck's theorem is usually stated as the isomorphism $\projlim_s\,H^s(X,N)\otimes_AA/I^s\isoto\projlim_sH^n(X_s,N/I^sN)$, but a quick examination of the proof in EGA shows that the stronger isomorphism of pro $A$-modules holds.
\end{proof}

Combining Grothendieck's formal function theorem with the already-established \linebreak degree-wise continuity results proves the following formal function theorems for $\HH$, $\THH$, and $\TR^r$:

\begin{theorem}\label{theorem_degree_wise_cont_proper}
Let $A$ be a Noetherian, F-finite, finite Krull-dimensional $\bb Z_{(p)}$-algebra, $I\subseteq A$ an ideal, $X$ a proper scheme over $A$, and $v\ge1$. Then the canonical maps
\begin{enumerate}
\item $\{\HH_n(X;\bb Z/p^v)\otimes_AA/I^s\}_s\To\{\HH_n(X_s,\bb Z/p^v\bb Z)\}_s$
\item $\{\TR_n^r(X;\bb Z/p^v)\otimes_{W_r(A)}W_r(A/I^s)\}_s\To\{\TR_n^r(X_s;\bb Z/p^v)\}_s$
\end{enumerate}
are isomorphisms for all $n\ge0$, $r\ge1$.
\end{theorem}
\begin{proof}
(i): Exactly as in the proof of Theorem \ref{theorem_finite_gen_proper}, there is a bounded Zariski descent spectral sequence of finitely generated $A$-modules  \[E_2^{ij}=H^i(X,\cal\HH_{-j}(X;\bb Z/p^v))\implies\HH_{-i-j}(X;\bb Z/p^v).\] By Theorem \ref{theorem_Artin_Rees}(ii), we may base change by $A/I^\infty$ to obtain a bounded spectral sequence of pro $A$-modules \[E_2^{ij}(\infty)=\{H^i(X,\cal\HH_{-j}(X;\bb Z/p^v))\otimes_A A/I^s\}_s\implies\{\HH_{-i-j}(X;\bb Z/p^v)\otimes_AA/I^s\}_s.\]

There is also a bounded Zariski descent spectral sequence for $\HH$ associated to each scheme $X_s=X\times_AA/I^s$, for $s\ge1$, and these also assemble to a spectral sequence of pro $A$-modules \[^\prime E_2^{ij}(\infty)=\{H^i(X_s,\cal\HH_{-j}(X_s;\bb Z/p^v))\}_s\implies\{\HH_{-i-j}(X_s;\bb Z/p^v)\}_s.\] These two spectral sequences are compatible in that the $E(\infty)$ spectral sequence maps to the $^\prime E(\infty)$ spectral sequence, and so to complete the proof it is enough to show that the canonical map of pro $A$-modules on the second pages \[\{H^i(X,\cal\HH_{-j}(X;\bb Z/p^v))\otimes_A A/I^s\}_s\To\{H^i(X_s,\cal\HH_{-j}(X_s;\bb Z/p^v))\}_s\] is an isomorphism for all $i,j\ge0$. But this map factors as the composition
\begin{align*}
\{H^i(X,\cal\HH_{-j}(X;\bb Z/p^v))\otimes_A A/I^s\}_s&\\
\To\{H^i(X,\cal\HH_{-j}(X;&\bb Z/p^v)\otimes_{\roi_X}\roi_X/I^s\roi_X)\}_s\\
&\To\{H^i(X_s,\cal\HH_{-j}(X_s;\bb Z/p^v))\}_s,
\end{align*}
where the first arrow is an isomorphism by Grothendieck's formal functions theorem and Corollary \ref{corollary_coherence_of_THH_TR}, and the second arrow is an isomorphism because the map of pro sheaves $\{\cal\HH_{-j}(X;\bb Z/p^v)\otimes_{\roi_X}\roi_X/I^s\roi_X\}_s\to\{\cal\HH_{-j}(X_s;\bb Z/p^v)\}_s$ is an isomorphism by Theorem \ref{theorem_continuity}, since all the affine open subschemes of $X$ are Noetherian, F-finite $\bb Z_{(p)}$-algebras, by Lemma \ref{lemma_conditions_for_F-finite}.

(ii): The proof is very similar to that of part (i), just requiring some additional care when applying Grothendieck's formal functions theorem. Firstly, as in the proof of Theorem \ref{theorem_finite_gen_proper}, there is a bounded Zariski descent spectral sequence of finitely generated $W_r(A)$-modules \[F_2^{ij}=H^i(X,\cal\TR_{-j}^r(X;\bb Z/p^v))\implies \TR_{-i-j}^r(X;\bb Z/p^v).\] By Theorem \ref{theorem_base_change_by_Witt_rings}(i), we may base change by $W_r(A/I^\infty)$ to obtain a bounded spectral sequence of pro $W_r(A)$-modules \[\hspace{-7mm}F_2^{ij}(\infty)=\{H^i(X,\cal\TR_{-j}^r(X;\bb Z/p^v))\otimes_{W_r(A)}W_r(A/I^s)\}_s\implies \{\TR_{-i-j}^r(X;\bb Z/p^v)\otimes_{W_r(A)}W_r(A/I^s)\}_s.\]

There is also a bounded Zariski descent spectral sequence for $\TR^r$ associated to each scheme $X_s$, for $s\ge1$, and these assemble to a spectral sequence of pro $W_r(A)$-modules \[^\prime F_2^{ij}(\infty)=\{H^i(X_s,\cal\TR_{-j}^r(X_s;\bb Z/p^v))\}_s\implies \{\TR_{-i-j}^r(X_s;\bb Z/p^v)\}_s.\] The $F(\infty)$-spectral sequence maps to the $^\prime F(\infty)$-spectral sequence, and so to complete the proof it is enough to show that the canonical map of pro $W_r(A)$-modules on the second pages \[\{H^i(X,\cal\TR_{-j}^r(X;\bb Z/p^v))\otimes_{W_r(A)}W_r(A/I^s)\}_s\To \{H^i(X_s,\cal\TR_{-j}^r(X_s;\bb Z/p^v))\}_s,\tag{\dag}\] is an isomorphism for all $i\ge0$.

To do this we first transfer the problem to the scheme $W_r(X)$ by rewriting \[H^i(X,\cal\TR^r_{-j}(X;\bb Z/p^v))=H^i(W_r(X),R^{r-1}_*\cal\TR_{-j}^r(X;\bb Z/p^v))\] and \[H^i(X_s,\cal\TR_{-j}^r(X_s;\bb Z/p^v))=H^i(W_r(X_s),R^{r-1}_*\cal\TR_{-j}^r(X_s;\bb Z/p^v))\] since the cohomology of a sheaf is unchanged after pushing forward along a closed embedding \cite[Lem.~III.2.10]{Hartshorne1977}. We can now apply the same argument as in the part (i), by factoring (\dag) as a composition
\begin{align*}
\{H^i(W_r(X),R^{r-1}_*\cal\TR^r_{-j}(X;\bb Z/p^v))\}_s&\\
\To\{H^i(W_r(X), R^{r-1}_*\cal\TR_{-j}^r(X_s&;\bb Z/p^v)\otimes_{W_r(\roi_X)}W_r(\roi_X/I^s\roi_X)\}_s\\
&\To\{H^i(W_r(X_s),R^{r-1}_*\cal\TR_{-j}^r(X_s;\bb Z/p^v))\}_s.
\end{align*}
Since $W_r(X)$ is a proper scheme over $W_r(A)$, we may apply Grothendieck's formal functions theorem to the coherent (by Corollary \ref{corollary_coherence_of_THH_TR}) sheaf $R^{r-1}_*\cal\TR_{-j}^r(X_s;\bb Z/p^v)$ and ideal $W_r(I)\subseteq W_r(A)$ (whose powers are intertwined with $W_r(I^s)$, $s\ge1$, by Lemma \ref{lemma_witt_1}) to deduce that the first arrows is an isomorphism. The second arrows is an isomorphism for the same reason as in part (i): the underlying map of pro sheaves is an isomorphism by Theorem \ref{theorem_continuity}.
\end{proof}

The following is the scheme-theoretic analogue of Corollary \ref{corollary_pre_continuity}:

\begin{corollary}\label{corollary_pre_continuity_proper}
Let $A$ be a Noetherian, F-finite, finite Krull-dimensional $\bb Z_{(p)}$-algebra, $I\subseteq A$ an ideal, and $X$ a proper scheme over $A$. Then all of the following maps (not just the compositions) are isomorphisms for all $n\ge0$ and $v,r\ge1$:
\begin{align*}
&\HH_n(X;\bb Z/p^v)\otimes_A\hat A\To \HH_n(X\times_A\hat A;\bb Z/p^v)\To \projlim_s\HH_n(X_s;\bb Z/p^v)\\
&\TR^r(X;\bb Z/p^v)\otimes_{W_r(A)}W_r(\hat A)\To \TR^r(X\times_A\hat A;\bb Z/p^v)\To\projlim_s\TR^r(X_s;\bb Z/p^v)
\end{align*}
\end{corollary}
\begin{proof}
The proof is almost identical to that of Corollary \ref{corollary_pre_continuity}. Indeed, exactly as in that proof, we consider the canonical maps \[\HH_n(X;\bb Z/p^v)\otimes_A\hat A\To\projlim_s\HH_n(X;\bb Z/p^v)\otimes_AA/I^s\To\projlim_s\HH_n(X_s;\bb Z/p^v),\] where the first arrow is an isomorphism by Theorem \ref{theorem_finite_gen_proper} and standard commutative algebra, and the second arrow is an isomorphism by Theorem \ref{theorem_degree_wise_cont_proper}. Then, replacing $A$ by $\hat A$, we see that the canonical maps \[\HH_n(X\times_A\hat A;\bb Z/p^v)\To\projlim_s\HH_n(X\times_A{\hat A};\bb Z/p^v)\otimes_{\hat A} A/I^sA\To\projlim_s\HH_n(X_s;\bb Z/p^v)\] are also isomorphisms. This proves the claim for $\HH$, and the proof for $\TR^r$ is identical.
\end{proof}

We finally reach the scheme-theoretic analogue of Theorem \ref{theorem_continuity_in_complete_case}, namely spectral continuity of $\THH$, $\TR^r$, etc.:

\begin{theorem}\label{theorem_spectral_continuity_proper}
Let $A$ be a Noetherian, F-finite, finite Krull-dimensional $\bb Z_{(p)}$-algebra, $I\subseteq A$ an ideal, and $X$ be a proper scheme over $A$; assume that $A$ is $I$-adically complete. Then, for all $1\le r\le\infty$, the following canonical maps of spectra are all weak equivalences after $p$-completion:
\begin{align*}
\TR^r(X;p)&\To\holim_s\TR^r(X_s;p)\\
\TC^r(X;p)&\To\holim_s\TC^r(X_s;p)
\end{align*}
\end{theorem}
\begin{proof}
The proof is identical to that of Theorem \ref{theorem_continuity_in_complete_case}.  Indeed, fixing $1\le r<\infty $, the homotopy groups of $\holim_s\TR^r(X_s;\bb Z/p^v)$ fit into short exact sequences \[0\to{\projlim_s}^1\TR^r_{n+1}(X_s;\bb Z/p^v)\to\pi_n(\holim_s\TR^r(X_s;\bb Z/p^v))\to \projlim_s\TR^r_n(X_s;\bb Z/p^v)\to 0,\] and Theorem \ref{theorem_degree_wise_cont_proper}(i) implies that ${\projlim_s}^1\TR^r_{n+1}(X_s)\cong{\projlim_s}^1\TR^r_{n+1}(X;\bb Z/p^v)\otimes_AA/I^s$, which vanishes because of the surjectivity of the transition maps in the latter pro abelian group. In conclusion, the natural map \[\pi_n(\holim_s\TR^r(X_s;\bb Z/p^v))\To \projlim_s\TR^r_n(A/I^s;\bb Z/p^v)\] is an isomorphism for all $n\ge0$. But since $A$ is already $I$-adically complete, Corollary \ref{corollary_pre_continuity_proper} states that $\TR^r_n(X;\bb Z/p^v)\to\projlim_s\TR^r_n(X_s;\bb Z/p^v)$ is also an isomorphism for all $n\ge0$. So the map $\TR^r(X;\bb Z/p^v)\to\holim_s\TR^r(X_s;\bb Z/p^v)$ induces an isomorphism on all homotopy groups, as required to prove the first weak equivalence.

The claims for $\TC^r$, $\TR^\infty=\TR$, and $\TC^\infty=\TC$ then follow since homotopy limits commute.
\end{proof}

As in the affine case, we now present corollaries of the previous results in the cases in which $p$ is either nilpotent or the generator of $I$; we begin with the nilpotent case:

\begin{corollary}\label{corollary_p_nilpotent_proper}
Let $A$ be a Noetherian, F-finite $\bb Z_{(p)}$-algebra, $X$ a proper scheme over $A$, and $n\ge0$, $r\ge1$; assume $p$ is nilpotent in $A$. Then:
\begin{enumerate}
\item $\HH_n(X)$ and $\THH_n(X)$ are finitely generated $A$-modules.
\item $\TR_n^r(X;p)$ is a finitely generated $W_r(A)$-module.
\end{enumerate}
Now let $I\subseteq A$ be an ideal. Then all of the following maps are isomorphisms for all $n\ge0$, $r\ge1$:
\begin{itemize}
\item[] $\{\HH_n(X)\otimes_AA/I^s\}_s\To\{\HH_n(X_s)\}_s$,
\item[] $\{\TR_n^r(X;p)\otimes_{W_r(A)}W_r(A/I^s)\}_s\To\{\TR_n^r(X_s;p)\}_s$,
\item[] $\HH_n(X)\otimes_A\hat A\To \HH_n(X\times_A\hat A)\To \projlim_s\HH_n(X_s)$,
\item[] $\TR^r_n(X;p)\otimes_{W_r(A)}W_r(\hat A)\To \TR^r_n(X\times_A\hat A;p)\To\projlim_s\TR^r_n(X_s;p)$.
\end{itemize}
Moreover, the maps of spectra in the statement of Theorem \ref{theorem_spectral_continuity_proper} are weak equivalences without $p$-completing.
\end{corollary}
\begin{proof}
As in the proof of Corollary \ref{corollary_nilpotent_continuity}, the groups $\HH_n(A)$, $\THH_n(A)$, and $\TR_n^r(A;p)$ are all bounded $p$-torsion. Hence the spectra are $p$-complete, and our finite generation and isomorphism claims follow from the already established versions with finite coefficients, namely Theorems \ref{theorem_finite_gen_proper} and \ref{theorem_degree_wise_cont_proper}, and Corollary \ref{corollary_pre_continuity_proper}.

Note that we have not assumed that $A$ has finite Krull dimension. This is because a Noetherian, F-finite $\bb Z_{(p)}$-algebra in which $p$ is nilpotent automatically has finite Krull dimension: indeed, it suffices to show that $A/pA$ has finite Krull dimension, and this follows from a theorem of E.~Kunz \cite[Prop.~1.1]{Kunz1976}.
\end{proof}

Next we consider the special case $I=pA$:

\begin{corollary}\label{corollary_proper_I_p}
Let $A$ be a Noetherian, F-finite, finite Krull-dimensional $\bb Z_{(p)}$-algebra, $X$ a proper scheme over $A$, and $v\ge1$. Then the following canonical maps are isomorphisms for all $n\ge0$, $r\ge1$:
\begin{itemize}
\item[] $\HH_n(X;\bb Z/p^v)\To\{\HH_n(X\times_AA/p^sA;\bb Z/p^v)\}_s$
\item[] $\TR_n^r(X;\bb Z/p^v)\To\{\TR_n^r(X\times_AA/p^sA;\bb Z/p^v)\}_s$
\item[] $\TC_n^r(X;\bb Z/p^v)\To\{\TC_n^r(X\times_AA/p^sA;\bb Z/p^v)\}_s$ 
\end{itemize}
Moreover, for all $1\le r\le\infty$, all of the following maps (not just the compositions) of spectra are weak equivalences after $p$-completion:
\begin{itemize}
\item[] $\TR^r(X;p)\To\TR^r(X\times_AA_p^\comp;p)\To\holim_s\TR^r(X\times_AA/p^sA;p)$
\item[] $\TC^r(X;p)\To\TC^r(X\times_AA_p^\comp;p)\To\holim_s\TC^r(X\times_AA/p^sA;p)$
\end{itemize}
\end{corollary}
\begin{proof}
Taking $I=pA$, these claims are proved from Theorem \ref{theorem_degree_wise_cont_proper} by the exact same argument by which Corollary \ref{corollary_p_I} was deduced from Theorem \ref{theorem_continuity}.
\end{proof}

For the sake of reference we also explicitly mention the following corollary of our results, as it will be used in forthcoming applications to the deformation theory of algebraic cycles in characteristic $p$:

\begin{corollary}
Let $A$ be a Noetherian, F-finite, finite Krull-dimensional $\bb Z_{(p)}$-algebra, $I\subseteq A$ an ideal, and $X$ be a proper scheme over $A$; assume that $A$ is $I$-adically complete. The the trace map from $K$-theory to topological cyclic homology yields a homotopy cartesian square of spectra:
\xysquare{\holim_sK(X\times_AA/I^s)_p^\comp}{K(X\times_AA/I)_p^\comp}{\TC(X;p)_p^\comp}{\TC(X\times_AA/I;p)_p^\comp}{->}{->}{->}{->}
\end{corollary}
\begin{proof}
The scheme-theoretic version (which follows from the affine version using Geisser--Hesselholt's Zariski descent \cite{GeisserHesselholt1999}) of R.~McCarthy's trace map isomorphism for nilpotent ideals \cite{McCarthy1997} \cite[Thm.~VII.0.0.2]{Dundas2013} states that, for any $s\ge1$, the following diagram becomes homotopy cartesian after $p$-completion:
\[\xymatrix{
K(X\times_AA/I^s)\ar[r]\ar[d]^{tr} & K(X\times_AA/I)\ar[d]^{tr}\\
\TC(X\times_AA/I^s;p)\ar[r] & \TC(X\times_AA/I;p)
}\] The corollary follows by taking $\op{holim}_s$, and applying Theorem \ref{theorem_spectral_continuity_proper} to the bottom left entry.
\end{proof}

\begin{appendix}
\section{Finite generation of $p$-completions}\label{subsection_p_complete}
\comment{
Given a ring $A$ and module $M$, define $\HH_n(A,M;\bb Z_p)$ to be the homotopy groups of the $p$-completion of the simplicial abelian group defining the (derived) Hochschild homology with coefficients in $M$.

Given another $A$-module $N$, there is a universal coefficient spectral sequence \[\Tor^A_i(N,\HH_j(A,M;\bb Z_p))\implies \HH_{i+j}(A,N\otimes_AM;\bb Z_p).\] At least if $A_p^\comp$ is flat over $A$, e.g. if $A$ is Noetherian, it is nice to note that the lhs is the same as $\Tor_i^{A_p^\comp}(N\otimes_AA_p^\comp,\HH_j(A,M;\bb Z_p))$.

The analogue of the spectral sequence from Proposition \ref{proposition_Bjorns_SS} is \[\HH_i(B,\Tor_j^A(B,M);\bb Z_p)\implies\HH_{i+j}(A,M;\bb Z_p).\]
}

In this appendix we explain how our finite generation results with finite coefficients lead to similar finite generation statements after $p$-completing. We claim no originality for these results, but could not find such algebraic statements summarised in the literature.

Our main goal is the following:

\begin{proposition}\label{proposition_to_prove_finite_gen_p_adically}
Let $A$ be a commutative, Noetherian ring, and $p$ a prime number.
\begin{enumerate}
\item Suppose $R_\blob$ is a simplicial ring and that $A\to \pi_0(R_\blob)$ is a ring homomorphism such that the homotopy groups of $R\dotimes_{\bb Z}\bb Z/p\bb Z$ are finitely generated $A$-modules; then $\pi_n(R_\blob;\bb Z_p)$ is naturally a finitely generated $A_p^\comp$-module for all $n\ge0$.
\item Suppose $X$ is a connective ring spectrum and that $A\to \pi_0(R_\blob)$ is a ring homomorphism such that the homotopy groups of $X\wedge S/p$ are finitely generated $A$-modules; then $\pi_n(M_\blob;\bb Z_p)$ is naturally a finitely generated $A_p^\comp$-module for all $n\ge0$.
\end{enumerate}
\end{proposition}

The proof of the proposition will be obtained through several algebraic lemmas. We first recall that that homotopy groups of a $p$-completion fit into a short exact sequence \[0\To\op{Ext}^1_{\bb Z}(\bb Q_p/\bb Z_p,\pi_n(-))\To \pi_n(-;\bb Z_p)\To\Hom_{\bb Z}(\bb Q_p/\bb Z_p,\pi_{n-1}(-)))\To 0,\] where the following identifications are often useful:
\begin{itemize}
\item $\bb Q_p/\bb Z_p=\bb Z[\tfrac{1}{p}]/\bb Z$, often denoted by $\bb Z_{p^\infty}$.
\item For any abelian group $M$, there is a short exact sequence \[0\to{\projlim_r}^1M[p^r]\to \op{Ext}^1_{\bb Z}(\bb Q_p/\bb Z_p,M)\to M_p^\comp\to 0,\] where $M[p^r]$ denotes the $p^r$-torsion of $M$ (the transition maps are multiplication by $p$), and $M_p^\comp=\projlim_r M/p^rM$ denotes its $p$-adic completion
\item For any abelian group $M$, the Tate module $\projlim_rM[p^r]$ is naturally isomorphic to $\Hom_{\bb Z}(\bb Q_p/\bb Z_p,M)$.
\end{itemize}

\begin{lemma}\label{lemma_p_homotopy_1}
Let $A$ be a commutative, Noetherian ring, and $M$ an $A$-module. Consider the statements:
\begin{enumerate}
\item $M$ is flat or finitely generated over $A$.
\item The $p$-power torsion of $M$ is bounded, i.e.~there exists $c\ge1$ such that any $p$-power torsion element of $M$ is killed by $p^c$.
\item The pro $A$-module $\{M[p^r]\}_r$ vanishes.
\item $\projlim_r^1M[p^r]$ and $\Hom_\bb Z(\bb Q_p/\bb Z_p,M)$ are zero.
\end{enumerate}
Then (i)$\implies$(ii)$\implies$(iii)$\implies$(iv).
\end{lemma}
\begin{proof}
Suppose first that $M$ is a finitely generated $A$-module. Then $M$ satisfies the ascending chain condition on submodules, so the chain $M[p]\subseteq M[p^2]\subseteq M[p^3]\subseteq\cdots$ is eventually constant, meaning exactly that all $p$-power torsion in $M$ is killed by $p^c$ for some fixed $c\ge1$. This proves (ii) for finitely generated $M$.

In particular, taking $M=A$, there exists $c\ge1$ such that $A[p^r]=A[p^c]$ for all $r\ge c$; in other words, the sequence \[0\to A[p^c]\to A\xto{\times p^r}A\] is exact. So, if $M$ is a flat $A$-module then tensoring by $M$ reveals that \[0\to A[p^c]\otimes_A M\to M\xto{\times p^r}M\] is also exact, whence $p^cM[p^r]=0$ for all $r\ge c$; this proves (ii) for flat $M$.

(ii)$\Rightarrow$(iii): Let $c\ge1$ be as in part (ii). Then, for each $r\ge1$, the transition map $M[p^{r+c}]\xto{\times p^c}M[p^r]$ is zero, as required to prove the vanishing of $\{M[p^r]\}_r$.

(iii)$\Rightarrow$(iv): This is immediate, using the identification $\Hom_{\bb Z}(\bb Q_p/\bb Z_p,M)=\projlim_rM[p^r]$.
\comment{
(i): The image of a homomorphism $\bb Q_p/\bb Z_p\to M$ lands inside the submodule of infinitely $p$-divisible elements of $M[p^\infty]$; by part (i), this is zero. Alternatively, note that $\Hom_\bb Z(\bb Q_p/\bb Z_p,M)=\projlim_rM[p^r]$ and apply (iii).

The image of a homomorphism $\bb Q_p/\bb Z_p\to M$ lands inside the $A$-module \[N:=\{x\in M:x\in p^rM\mbox{ for all }r\ge1\mbox{, and $x$ is killed by a power of }p\},\] so it is enough to show $N=0$. Suppose $x\in N$. Since $x\in\bigcap_{r\ge1}p^rM$, a consequence of the Artin--Rees lemma states that $\al x=0$ for some $\al\in1+pA$. But $\al^{p^s}\in1+p^{s+1}A$ for all $s\ge1$, and so $\al^{p^s} x=x$ for $s\gg0$ since $x$ is killed by a power of $p$; hence $x=0$, as desired.}
\end{proof}

\comment{
\begin{leftbar} Do we need this?
\begin{corollary}\label{corollary_p_local_groups_in_fin_gen_case}
Let $A$ be a commutative, Noetherian ring, and let $M_\blob$ be a simplicial $A$-module each of whose homotopy groups is either a flat or a finitely generated $A$-module. Then \[\pi_n(M_\blob;\bb Z_p)\cong \pi_n(M_\blob)_p^\comp\] for all $n\ge0$. In the case when all the homotopy groups are finitely generated $A$-modules, it is moreover true that \[\pi_n(M_\blob;\bb Z_p)\cong\pi_n(M_\blob)\otimes_AA_p^\comp\] for all $n\ge0$.
\end{corollary}
\begin{proof}
Apply the previous lemma and the aforementioned short exact sequences describing $\pi_n(M_\blob;\bb Z_p)$ to prove the first claim.

The second claim is then a standard result from commutative algebra: for any ideal $I\subseteq A$ and finitely generated $A$-module $M$, one has $\projlim_rM/I^rM\cong M\otimes_A\hat A$, where $\hat A:=\projlim_rA/I^r$.
\end{proof}
\end{leftbar}
}

\begin{lemma}
Let $A$ be a commutative, Noetherian ring, and $M$ an $A$-module.
\begin{enumerate}
\item If $M/pM$ is a finitely generated $A$-module, then $\Ext_\bb Z^1(\bb Q_p/\bb Z_p,M)$ is a finitely generated $A_p^\comp$-module. 
\item If $M[p]$ is a finitely generated $A$-module, then $\Hom_\bb Z(\bb Q_p/\bb Z_p,M)$ is a finitely generated $A_p^\comp$-module.
\end{enumerate}
\end{lemma}
\begin{proof}
(i): There exists a finitely generated $A$-submodule $N\subseteq M$ such that $N/pN\to M/pM$ is surjective; let $\Lambda=M/N$, which is $p$-divisible. The long exact $\op{Ext}_\bb Z(\bb Q_p/\bb Z_p,-)$ sequence for $0\to N\to M\to \Lambda\to 0$ contains \[\op{Ext}_\bb Z^1(\bb Q_p/\bb Z_p,N)\to \op{Ext}_\bb Z^1(\bb Q_p/\bb Z_p,M)\to \op{Ext}_\bb Z^1(\bb Q_p/\bb Z_p,\Lambda),\] and we claim that the final term vanishes. Indeed, it is equivalent to show that both $\projlim_r^1\Lambda[p^r]$ and $\Lambda_p^\comp$ vanish, and this easily follows from the $p$-divisible of $\Lambda$.

Moreover, it follows from Lemma \ref{lemma_p_homotopy_1} that $\projlim_r^1N[p^r]=0$, whence $\op{Ext}_\bb Z^1(\bb Q_p/\bb Z_p,N)=N_p^\comp$, which is a finitely generated $A_p^\comp$-module; since $N_p^\comp$ surjects onto $\op{Ext}_\bb Z^1(\bb Q_p/\bb Z_p,M)$, we deduce that the latter is also finitely generated.

(ii): The ``evaluation at $\tfrac{1}{p}$'' map induces an injection \[\Hom_\bb Z(\bb Q_p/\bb Z_p,M)\otimes_\bb Z\bb Z/p\bb Z\into M[p],\] whence the left side is finitely generated over $A_p^\comp$. Applying part (i) to the module $\Hom_\bb Z(\bb Q_p/\bb Z_p,M)$, we therefore deduce that $\Ext_\bb Z^1(\bb Q_p/\bb Z_p,\Hom_\bb Z(\bb Q_p/\bb Z_p,M))$ is a finitely generated $A_p^\comp$-module, hence that its quotient $\Hom_\bb Z(\bb Q_p/\bb Z_p,M)_p^\comp$ is also a finitely generated $A_p^\comp$-module. But $\Hom_\bb Z(\bb Q_p/\bb Z_p,M)=\projlim_rM[p^r]$ is evidently already $p$-adically complete, so that $\Hom_\bb Z(\bb Q_p/\bb Z_p,M)_p^\comp=\Hom_\bb Z(\bb Q_p/\bb Z_p,M)$.
\end{proof}

\begin{proof}[Proof of Proposition \ref{proposition_to_prove_finite_gen_p_adically}]
It is sufficient to prove (ii), since (i) is then obtained by applying the Eilenberg--Maclane construction.

The $p$-completion $X_p^\comp$ is again a ring spectrum and so the groups $\pi_n(X;\bb Z_p)$ are naturally $\pi_0(X;\bb Z_p)$-modules. Moreover, there is a canonical ring homomorphism \[A_p^\comp=\Ext_\bb Z^1(\bb Q_p/\bb Z_p;A)\To \Ext_\bb Z^1(\bb Q_p/\bb Z_p;\pi_0(X))=\pi_0(X;\bb Z_p),\] where the first equality follows from Lemma \ref{lemma_p_homotopy_1} with $M=A$, and the second equality follows from the fact that $X$ is connected.

Hence $\pi_n(X;\bb Z_p)$ is naturally an $A_p^\comp$-module, and \[0\To\op{Ext}^1_{\bb Z}(\bb Q_p/\bb Z_p,\pi_n(X))\To \pi_n(X;\bb Z_p)\To\Hom_{\bb Z}(\bb Q_p/\bb Z_p,\pi_{n-1}(X)))\To 0\tag{\dag}\] is an exact sequence of $A_p^\comp$-modules. By assumption $\pi_n(X)\otimes_\bb Z\bb Z/p\bb Z$ and $\pi_n(X)[p]$ are finitely generated $A$-modules for all $n\ge0$, so the previous lemma implies that the outer terms of (\dag) are finitely generated $A_p^\comp$-modules; this completes the proof.
\end{proof}

\comment{
(i): Letting $\bb A$ denote the constant simplicial ring obtained from $A$, it is evident that $M_\blob^\comp$ is an $\bb A^\comp$-module, whence the groups $\pi_n(M_\blob;\bb Z_p)$ are naturally $\pi_0(\bb A;\bb Z_p)$-modules. But it follows from Corollary \ref{corollary_p_local_groups_in_fin_gen_case} with $M_\blob=\bb A$ that $\pi_0(\bb A;\bb Z_p)=A_p^\comp$; this describes the natural $A_p^\comp$-module structure on $\pi_n(M_\blob;\bb Z_p)$.

From the usual short exact sequences \[0\to \pi_n(M_\blob)\otimes_\bb Z\bb Z/p\bb Z\to\pi_n(M_\blob\dotimes_\bb Z\bb Z/p\bb Z)\to \pi_{n-1}(M_\blob)[p]\to0\] we deduce that $\pi_n(M_\blob)\otimes_\bb Z\bb Z/p\bb Z$ and $\pi_n(M_\blob)[p]$ are finitely generated $A$-modules for all $n\ge0$. So the claim follows from the short exact sequences \[0\to\op{Ext}^1_{\bb Z}(\bb Q_p/\bb Z_p,\pi_n(M_\blob))\to \pi_n(M_\blob;\bb Z_p)\to\Hom_{\bb Z}(\bb Q_p/\bb Z_p,\pi_{n-1}(M_\blob))\to 0\] and the previous lemma.
}

\end{appendix}

\bibliographystyle{acm}
\bibliography{../Bibliography}

\noindent Bj\o rn Ian Dundas\hfill Matthew Morrow\\
University of Bergen\hfill Mathematisches Institut\\
Johs.~Brunsgt.~12\hfill Universit\"at Bonn\\
N-5008 Bergen\hfill Endenicher Allee 60\\
Norway\hfill 53115 Bonn, Germany\\
\url{dundas@math.uib.no}\hfill{\tt morrow@math.uni-bonn.de} \\
\url{http://www.uib.no/People/nmabd/}\hfill \url{http://www.math.uni-bonn.de/people/morrow/}

\end{document}